\documentclass[dvipsnames, a4paper, 11pt]{article}

\usepackage{hyperref}
\usepackage[utf8]{inputenc}
\usepackage[T1]{fontenc}
\usepackage{amsmath}
\usepackage{amsthm}
\usepackage{amsfonts}
\usepackage{amssymb}
\usepackage{graphicx,caption,subcaption}
\usepackage{fancybox}
\usepackage{enumitem}
\usepackage{soul}
\usepackage{mathtools}
\usepackage{lmodern}
\usepackage{stmaryrd}
\usepackage{multirow}
\usepackage{multicol}
\usepackage[usenames,dvipsnames]{xcolor}
\usepackage{tabularx}
\usepackage{array}
\usepackage{colortbl}
\usepackage{varwidth}
\PassOptionsToPackage{dvipsnames,svgnames}{xcolor}     
\usepackage{wrapfig}
\usepackage[explicit,calcwidth]{titlesec}
\usepackage{fancyhdr}
\usepackage[left=2.1cm, right=2.1cm, top=2cm, bottom=2cm]{geometry}
\usepackage{esint}
\usepackage{tocloft}
\usepackage[skip=8pt plus1pt, indent=0pt]{parskip}
\usepackage{twoopt}
\usepackage{authblk}
\usepackage{accents}
\usepackage{mathrsfs}
\usepackage{adjustbox} %
\hypersetup{ colorlinks=true, %
    linktoc=all,     %
    linkcolor=blue,  %
}
\PassOptionsToPackage{unicode}{hyperref}
\PassOptionsToPackage{naturalnames}{hyperref}
\usepackage{float}
\usepackage[nameinlink, capitalise]{cleveref}
\usepackage[ruled,vlined,linesnumbered]{algorithm2e}
\crefname{algocf}{Algorithm}{Algorithms} %

\usepackage{todonotes}
\setuptodonotes{fancyline, color=red!20, shadow, inline}

\usepackage{silence}

\usepackage[bbgreekl]{mathbbol}
\DeclareSymbolFont{bbold}{U}{bbold}{m}{n}
\DeclareSymbolFontAlphabet{\mathbbold}{bbold}
\DeclareSymbolFontAlphabet{\mathbb}{AMSb}%

\newtheorem{theorem}{Theorem}[section]
\newtheorem*{theorem*}{Theorem}
\newtheorem{corollary}[theorem]{Corollary}
\newtheorem{lemma}[theorem]{Lemma}
\newtheorem*{lemma*}{Lemma}

\newtheorem{remark}[theorem]{Remark}
\newtheorem{prop}[theorem]{Proposition}

\newtheorem{definition}[theorem]{Definition}

\newcommand{\N}[0]{\mathbb{N}}

\newcommand{\Q}[0]{\mathbb{Q}}
\newcommand{\R}[0]{\mathbb{R}}

\newcommand{\K}[0]{\mathcal{K}}

\newcommand{\U}[0]{\mathcal{U}}
\renewcommand{\P}[0]{\mathbb{P}}

\newcommand{\Int}[2]{\displaystyle\int_{#1}^{#2}}

\newcommand{\Sum}[2]{\displaystyle\sum\limits_{#1}^{#2}}

\newcommand{\Reu}[2]{\displaystyle\bigcup\limits_{#1}^{#2}}
\newcommand{\E}[1]{\mathbb{E}\left[#1\right]}

\newcommand{\dr}[3]{\cfrac{\partial ^{#2} {#3}}{\partial {#1}^{#2}}}

\newcommand{\dd}[0]{\mathrm{d}}
\definecolor{mybluei}{RGB}{0,173,239}

\renewcommand{\epsilon}{\varepsilon}

\DeclareMathOperator{\argmin}{argmin}

\DeclareMathOperator{\argmax}{argmax}

\DeclareMathOperator{\conv}{conv}
\DeclareMathOperator{\supp}{supp}
\newcommand{\W}[0]{\mathrm{W}}

\newcommand{\interior}[1]{\mathring{#1}}

\newcommand{\app}[4]{\left\lbrace\begin{array}{ccc}
   #1 & \longrightarrow & #2 \\
   #3 & \longmapsto & #4 \\
\end{array} \right.}

\newcommand{\oll}[1]{\overline{#1}}

\renewcommand{\O}{\mathcal{O}}
\newcommand\eqquestion{\stackrel{\mathclap{\normalfont\mbox{?}}}{=}}

\newcommand{\step}[2]{\textrm{---} \textit{Step #1}: #2}

\newcommand{\pot}{\varphi}
\newcommand{\gradlip}{L}
\newcommand{\strcvx}{\ell}

\newcommand{\Funs}{\mathcal{F}_{\mathcal{E}, \gradlip, \strcvx}}
\newcommand{\FunsRd}{\mathcal{F}_{\gradlip, \strcvx}}
\newcommand{\Tc}{\mathcal{T}_c}
\newcommand{\TC}{\mathcal{T}_C}
\newcommand{\linv}[1]{#1^{\rightarrow}}
\newcommand{\rinv}[1]{#1^{\leftarrow}}

\newcommand{\Leb}{\mathscr{L}}
\newcommand{\X}{\mathcal{X}}
\newcommand{\Y}{\mathcal{Y}}
\newcommand{\T}{\mathcal{T}}
\newcommand{\GNN}{G_{\mathrm{NN}}}

\newcommand{\RKHS}{\mathcal{H}}

\newcommand{\I}{\mathbf{I}}

\newcommand{\blue}[1]{{\color{BlueViolet}#1}}  %
\renewcommand{\blue}[1]{#1}

\makeatletter
\renewcommand*{\fps@figure}{h}  %
\makeatother

\title{Constrained Approximate Optimal Transport Maps}
\author[1]{Eloi Tanguy}
\author[2]{Agnès Desolneux}
\author[1]{Julie Delon}
\affil[1]{Universit\'e Paris Cit\'e, CNRS, MAP5, F-75006 Paris, France}
\affil[2]{Centre Borelli, CNRS and ENS Paris-Saclay, F-91190 Gif-sur-Yvette, France}
\date{12th March 2025}
  
\begin{document}
\maketitle

\begin{abstract}
	We investigate finding a map $g$ within a function class $G$ that minimises an
Optimal Transport (OT) cost between a target measure $\nu$ and the image by $g$
of a source measure $\mu$. This is relevant when an OT map from $\mu$ to $\nu$
does not exist or does not satisfy the desired constraints of $G$. We address
existence and uniqueness for generic subclasses of $L$-Lipschitz functions,
including gradients of (strongly) convex functions and typical Neural Networks.
We explore a variant that approaches a transport plan, showing equivalence to a
map problem in some cases. For the squared Euclidean cost, we propose
alternating minimisation over a transport plan $\pi$ and map $g$, with the
optimisation over $g$ being the $L^2$ projection on $G$ of the barycentric
mapping $\overline{\pi}$. In dimension one, this global problem equates the
$L^2$ projection of $\overline{\pi^*}$ onto $G$ for an OT plan $\pi^*$ between
$\mu$ and $\nu$, but this does not extend to higher dimensions. We introduce a
simple kernel method to find $g$ within a Reproducing Kernel Hilbert Space in
the discrete case. We present numerical methods for $L$-Lipschitz gradients of
$\ell$-strongly convex potentials\blue{, and study the convergence of Stochastic
Gradient Descent methods for Neural Networks. We finish with an illustration  on
colour transfer, applying learned maps on new images, and showcasing outlier
robustness.}

\end{abstract}

\tableofcontents

\section{Introduction}
Let $\mu$ and $\nu$ denote two probability distributions on two (potentially
different) measurable spaces $\X$ and $\Y$. Many problems in applied fields can
be written under the form
\begin{equation}
  \label{eq:pb_general}
  \underset{g \in G}{\inf}\ \mathcal{D}(g\# \mu,\nu),
\end{equation}
where $\#$ denotes the \textit{push-forward} operation\footnote{The image
measure $g\#\mu$ is defined as the law of $g(X)$ for $X$ a random variable of
law $\mu$, or more abstractly by $g\#\mu(B) = \mu(g^{-1}(B))$ for any Borel set
$B \subset \mathcal{Y}$.}, $\mathcal{D}$ is a non-negative discrepancy (such as
a distance metric or a $\phi$-divergence) measuring the similarity between $g\#
\mu$ and $\nu$, and $G$ is a set of acceptable functions from $\mathcal{X}$ to
$\mathcal{Y}$. Under appropriate assumptions on $\mathcal{D}$, this problem can
be interpreted as a projection of $\nu$ on the set $G\# \mu :=\{g\#\mu,g\in G\}$
for the discrepancy $\mathcal{D}$. In this paper, we focus on cases where $\nu$
cannot be written as $g\#\mu$ for $g\in G$.\footnote{Obviously, if $\nu$ belongs
to $G\#\mu$, the problem is trivial and the infimum in~\cref{eqn:CATM} is $0$.}

\blue{In the highly popular field of generative modelling, the target
distribution is usually an empirical distribution composed of $m$ samples, $\nu
= \frac 1 m \sum_{i=1}^m \delta_{x_i}$, $\mu$ is an easy-to-sample latent
distribution (for instance a Gaussian distribution), and the set $G =
\{g_\theta,\;\theta \in \Theta\}$ generally  denotes functions represented by a
specific neural network architecture. The goal is to find the parameter $\theta$
such that $\mu_\theta:=g_\theta\#\mu$ fits $\nu$ as well as possible. Models
taking this form are often called push-forward generative
models~\cite{salmona2022push}, and include Variational Auto-Encoders
(VAEs)~\cite{kingma2013auto}, Generative Adversarial Networks
(GANs)~\cite{goodfellow2014generative}, Normalising
flows~\cite{rezende2015variational} and even Diffusion
Models~\cite{song2020score}, which can be reinterpreted as indirect push-forward
generative models~\cite{salmona2022push}. In these works, the discrepancy
$\mathcal{D}$ is often chosen as the Kullback-Leibler divergence, as it is the
case for traditional GANs and VAEs, or as the Wasserstein distance, like in
Wasserstein-GANs~\cite{arjovsky2017wasserstein}. The discrepancy
$\mathcal{D}(g_\theta\# \mu,\nu)$ is minimised in $\theta$, for instance by
using sophisticated versions of stochastic gradient descent. In such problems,
it is clear that $g_\theta\# \mu$ does not target  exactly $\nu$, since it would
mean that the model has only learned to reproduce existing samples, and not to
create new ones. This is possible because the expressivity of neural networks is
limited, but also because the training steps usually impose regularity
properties on $g_\theta$ and constrain its Lipschitz constant in order to
increase its robustness~\cite{scaman2019lipschitz,fazlyab2019efficient} or
stabilise its training \cite{miyato2018spectral}. It is therefore natural to
wonder to what extent the optimisation of such discrepancies with regularity
constraints on the set of functions $G$ is well-posed, depending on the choice
of $\mathcal{D}$, and what this means in practice.}

In an Euclidean setting, another example of~\cref{eq:pb_general} appears when we
need to compare two distributions $\mu$ and $\nu$ potentially living in spaces
of different dimensions, or when invariance to geometric transformations is
required (for problems such as shape matching or word embedding). In such cases,
it is usual to choose $G$ as a well chosen set of linear or affine embeddings
(such as matrices in the Stiefel manifold if the space dimension is different
between $\X$ and $\Y$). For instance, this idea underpins several sets of works
introducing global invariances in optimal
transport~\cite{alvarez2019towards,salmona2023gromov}.

In both of the previous examples, $G$ is parametrised by a set $\Theta$ of
parameters which is potentially extremely large (for neural networks) but of
finite dimension.   
Alternatively, the set of functions $G$ can be much more complex and
characterised by regularity or convexity assumptions, the problem becoming
non-parametric. This is typically the case in the field of optimal
transport~\cite{villani, santambrogio2015optimal}.
Given $\mu$ and $\nu$ probability measures on respective Polish spaces
$\mathcal{X}$ and $\mathcal{Y}$, Monge's Optimal Transport consists in finding a
Transport map $T$ such that $T\#\mu = \nu$ and which minimises a given
displacement cost. \blue{From a theoretical standpoint, the existence of
(unconstrained) Monge maps has been widely studied
\cite{brenier1991polar,gangbo1996geometry,pratelli2007equality}, under
regularity assumption for $\mu$.} When there is no map $T$ such that $T\#\mu =
\nu$ (for example, if $\mu$ is discrete and if $\nu$ is not), or when the map
solution does not meet the regularity requirements for some given practical
application,
it makes sense instead to solve problems of the form of \cref{eq:pb_general},
with $\mathcal{D}$ a Wasserstein distance and $G$ a set of functions with
acceptable regularity. For instance, as studied in~\cite{paty2020regularity},
$G$ can be composed of functions $g=\nabla \phi$ with $\phi$
$\strcvx$-strongly-convex with an $\gradlip$-Lipschitz gradient. For cases where
$\mu$ is discrete, this formulation also overcomes a classic shortcoming of
numerical optimal transport approaches, which usually compute solutions which
are only defined on the support of $\mu$. If a machine learning algorithm
requires the computation of the transport of new inputs, the map must be either
recomputed, or an approximation of the previous map must be defined outside of
the support of $\mu$. Several solutions have been proposed in the literature to
solve this problem~\cite{de2021consistent,black2020fliptest,manole2021plugin,
paty2020regularity,pooladian2024entropicestimationoptimaltransport,seguy2017large,perrot2016mapping},
and some of them~\cite{manole2021plugin,paty2020regularity} consists in
solving~\cref{eq:pb_general} with an appropriate set of functions $G$.
Consistency and asymptotic properties of such estimators are also the subject of
several of these
works~\cite{de2021consistent,manole2021plugin,hutter2021minimax}.

\blue{For the sake of legibility and to avoid excessive technicality, we focus
on the case where the target space is $\R^d$, however it is possible to extend
our considerations to a target space $\Y$ which is a Polish space verifying the
Heine-Borel property (i.e. that any bounded and closed set is a compact set),
which in particular allows the case where $\Y$ is a connected and complete
Riemannian manifold (in which case the Heine-Borel property follows from the
Hopf-Rinow Theorem, see \cite{do1992riemannian} Theorem 2.8). Similarly, the
problem naturally extends to the case where the codomain of the maps $g$ and the
target measure $\nu$ are different spaces $\Y, \Y'$.}

\textbf{OT discrepancies.} In this paper, we focus on problems of the
form~\cref{eq:pb_general} when $\mathcal{D}$ is chosen as an optimal transport
discrepancy for a general ground cost $c$. We recall that if $\X$ and $\Y$ are
two Polish spaces, the Optimal Transport cost between two measures $\nu_1 \in
\mathcal{P}(\X)$  and $\nu_2 \in \mathcal{P}(\Y)$ for a ground cost function
$c:\X\times \Y \longrightarrow \mathbb{R}_+$ is defined by the following
optimisation problem
\begin{equation}\label{eqn:Tc}
	\T_c(\nu_1, \nu_2) = \underset{\pi \in \Pi(\nu_1,
          \nu_2)}{\min}\ \int_{\X \times \Y} c(x, y) \dd\pi(x,y),
\end{equation}
where $\Pi(\nu_1, \nu_2)$ is the set of probability measures on $\X\times \Y$
whose first marginal is $\nu_1$ and second marginal is $\nu_2$\footnote{The fact
that the minimum is attained is a consequence of the direct method of calculus
of variations (see \cite{santambrogio2015optimal}, Theorem 1.7). The value of
$\T_c(\nu_1, \nu_2)$ may be $+\infty$, but a sufficient condition for
$\T_c(\nu_1, \nu_2)<+\infty$ (\cite{villani}, Remark 5.14) is that
$$\int_{\X \times \Y}c(x,y)\dd\nu_1(x)\dd\nu_2(y)<+\infty. $$}. Given this
method of quantifying the discrepancy between $g\#\mu$ and $\nu$,
\cref{eq:pb_general} becomes
\begin{equation}\label{eqn:CATM}
  \underset{g\in G}{\inf}\ \T_c(g\#\mu, \nu).
\end{equation}
In the case where the source measure is discrete and the target measure is
absolutely continuous, the Optimal Transport problem in \cref{eqn:CATM} is said
to be semi-discrete, and has a slightly more explicit expression (see
\cite{merigot2020ot_course} for a course on the matter). If we suppose in
addition that $c(x,y) = \|x-y\|_2^2$ and that the source measure weights are
uniform ($a_i = 1/n$), then \cref{eqn:CATM} is a constrained version the Optimal
Uniform Quantization problem studied thoroughly in
\cite{merigot2021bounds_approx_W2}.

\textbf{Existence of minimisers.} An important question regarding this
optimisation problem concerns the existence of minimisers, depending on the
ground cost $c$ and the set of functions $G$. While numerous works  in the
literature have focused on the convergence of optimisation algorithms (such as
stochastic gradient descent) to critical points for this kind of
problem~\cite{fatras2021minibatch}, the existence of minimisers has surprisingly
been little studied. We derive
in~\cref{thm:existence_G_c,thm:existence_G_c_parts} generic conditions to ensure
existence of such minimisers in $G$, and show counter-examples when these
conditions are not met. We also show that these conditions are satisfied for two
classes of functions, namely classes of $L$-Lipschitz functions which can be
written as gradient of $l$-strongly convex functions (recovering a result shown
in~\cite{paty2020regularity} as a particular case of
\cref{thm:existence_G_c_parts}), and classes of neural networks with Lipschitz
activation functions.  We also discuss uniqueness of the solutions, which is
usually not satisfied, and remains a difficult question without strong
assumptions on the set of functions $G$.

\textbf{Approximating a coupling.} In the field of optimal transport, a
particular setting where~\cref{eqn:CATM} is interesting is when we have access
to a non deterministic coupling $\pi$ solution of a regularised version of a
optimal transport between two probability measures $\mu $ and $\nu$. For
instance, the entropic optimal transport~\cite{peyre2019computational}, or the
mixture Wasserstein formulation~\cite{delon2020wasserstein} both yield optimal
plans $\pi$ which cannot be trivially written as optimal maps between $\mu$ and
$\nu$. For some applications, it can be interesting to approximate $\pi$ by
another transport plan supported by the graph of a function with possible
additional regularity assumptions. This can be done by approximating $\pi$ by
$(I,g)\#\mu$, with specific regularity properties on $g$, which is a particular
case of~\cref{eqn:CATM}, replacing $\nu$ by $\pi$ and $G$ by the set $H :=
\{(I,g),g\in G\}$.
In this specific setting, we show in \cref{sec:plan_problem} under which
conditions on the ground cost $c$ the solutions of this problem between plans
are equivalent to solutions of the original~\cref{eqn:CATM} when $\pi \in
\Pi(\mu,\nu)$. \blue{Numerical approaches seeking maps that approach barycentric
projections have been studied in \cite{seguy2017large,perrot2016mapping}.}

\textbf{Alternate minimisation.} Under appropriate assumptions, \cref{eqn:CATM}
can be rewritten as a minimisation problem over $\pi\in \Pi(\mu, \nu)$ and $g\in
G$:
\begin{equation}
\underset{g\in G}{\min}\ \underset{\pi\in\Pi(\mu, \nu)}{\min}\
\int_{\X\times \Y}c(g(x), y)\dd\pi(x,y).
\label{eq:alternate_min}
\end{equation}
This naturally leads to consider \cref{eqn:CATM} as an alternate minimisation
problem, that we study  in \cref{sec:alternate_minimisation} in the Euclidean
case when $c(x,y) = \|x-y\|_2^2$. More precisely, we show that ~\cref{eqn:CATM}
is strongly linked to the barycentric projection problem: when $\pi$ is fixed,
the solution $g$ minimising~\cref{eq:alternate_min} can be reinterpreted as the
$L^2$-projection of the barycentric projection of $\pi$ on the set $G$. In the
one-dimensional case, when $G$ is a subclass of increasing functions, this
yields an explicit solution to the problem (as it was shown
in~\cite{paty2020regularity} in a more specific case), and we show that this
explicit solution does not hold in dimension larger than $1$ by presenting a
counter-example.

\textbf{Outline of the paper.} In this work, we address problem \cref{eqn:CATM}
for large classes of functions $G$. In \cref{section:constrained}, we define the
problem and establish general conditions for the existence of solutions,
exploring examples involving gradients of convex functions and neural networks.
\cref{sec:alternate_minimisation} examines the link between \cref{eqn:CATM} and
a constrained barycentric projection problem, demonstrating an explicit solution
in one dimension and providing a counterexample in higher dimension.
\cref{sec:discrete} focuses on practical numerical methods to solve the
optimisation problems \blue{for Lipschitz gradients of strongly convex
potentials, kernel methods and Neural Networks. We conclude with an illustration
on colour transfer.}
\section{A Constrained Approximate Transport Map Problem}
\label{section:constrained}
\subsection{Problem Definition}

We consider $(\X, d_\X)$ a locally compact Polish space, and $\mu\in
\mathcal{P}(\X)$ a probability measure on $\X$. Our objective is to find a map
$g: \X \longrightarrow \R^d$ verifying the constraint $g\in G$ for some class of
functions $G \subset (\R^d)^\X$, such that the image measure $g\#\mu$ is "close"
to a fixed probability measure $\nu \in \mathcal{P}(\R^d)$, in the sense
of~\cref{eqn:CATM}.

Applying the definition of $\Tc$ directly (\cref{eqn:Tc}) yields the following
expression for \cref{eqn:CATM}:
\begin{equation}\label{eqn:CATM_before_cov}
	\underset{g\in G}{\inf}\  \underset{\pi \in \Pi(g\#\mu, \nu)}{\min}\
	\int_{\X\times\R^{d}} c(x, y) \dd\pi(x,y).
\end{equation}
The optimisation variable $g$ acts on the set of constraints of the Optimal
Transport problem, however thanks to a well-known "change of variables" result
(\cite{dumont22gromovmap} Lemmas 1 and 2 for a reference), we will be able to
reformulate \cref{eqn:CATM_before_cov}. In the following, we shall denote by
$\Pi_c^*(\nu_1, \nu_2)$ the set of minimisers of the optimal transport problem
\cref{eqn:Tc} between two measures $\nu_1$ and $\nu_2$. 

\begin{lemma}\label{lemma:plans_and_push_forwads} (\cite{dumont22gromovmap},
	Lemmas \blue{2.6 and 2.7}) Let $\mathcal{X}, \mathcal{Y}, \mathcal{X}',
	\mathcal{Y}'$ be Polish spaces. Let $g:
	\mathcal{X}\longrightarrow\mathcal{X}'$ and $h:
	\mathcal{Y}\longrightarrow\mathcal{Y}'$ two measurable maps and let
	$(\mu,\nu) \in \mathcal{P}(\mathcal{X})\times \mathcal{P}(\mathcal{Y})$.
	Consider two costs $c: \mathcal{X}\times \mathcal{Y} \longrightarrow \R$ and
	$c': \mathcal{X}'\times \mathcal{Y}'\longrightarrow \R$ such that $\forall
	(x,y)\in \mathcal{X}\times\mathcal{Y},\; c(x,y)=c'(g(x), h(y))$.
	\begin{itemize}
		\item For any $\gamma' \in \Pi(g\#\mu, h\#\nu)$, there exists $\gamma
		\in \Pi(\mu, \nu)$ such that $\gamma' = (g, h)\#\gamma$.
		\item We have $\Pi_{c'}^*(g\#\mu, h\#\nu) = (g,h)\#\Pi_c^*(\mu, \nu)$.
	\end{itemize}
\end{lemma}

Using \cref{lemma:plans_and_push_forwads}, the energy of the map problem
\cref{eqn:CATM} can be written as follows:
\begin{align}\label{eqn:energy_formulation}
	\Tc(g\#\mu, \gamma) = \underset{\pi \in \Pi(\mu, \nu)}{\inf}\ \int_{\X\times \R^d}c(g(x), y)\dd\pi(x, y).
\end{align}

In our study of the map problem \cref{eqn:CATM}, we will consider classes $G$
that are a subset of the $\gradlip$-Lipschitz functions. The first reason is
that with unbounded Lipschitz constants, the problem may not have a solution, as
we shall see in \cref{sec:no_lip_no_existence}. Moreover, there are multiple
practical considerations that lead to choosing functions with an upper-bounded
Lipschitz constant. To begin with, numerous practical models enforce this
condition, such as Wasserstein GANs \cite{arjovsky2017wasserstein}, and
diffusion models \cite{song2020score} (see also \cite{salmona2022push} Appendix
S2), furthermore most neural networks are Lipschitz (since typical
non-linearities are chosen as Lipschitz), and the control of the Lipschitz
constant \blue{is often used} as a regularisation method
\cite{scaman2019lipschitz}. 
In \cref{fig:1d_ssnb_map_illu}, we illustrate a solution of the map problem
using numerical methods introduced in \cref{sec:numerics_grad_convex}, for two
different values of $L$  (the Lipschitz constant of the maps $g$).

\begin{figure}[ht]
    \centering
    \begin{adjustbox}{valign=c}
        \begin{subfigure}[b]{0.45\textwidth}
            \centering
            \includegraphics[width=\textwidth]{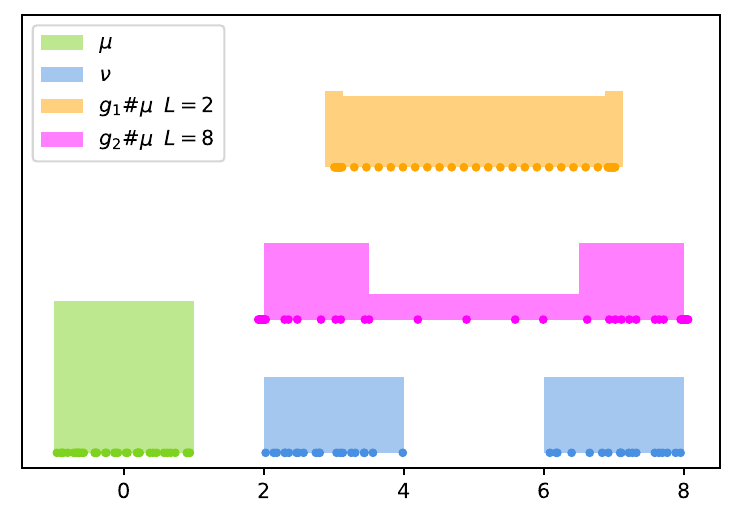}
            \caption{Illustration of the source, target and image measures using their samples and (approximate) densities.}
        \end{subfigure}
    \end{adjustbox}
    \hfill
    \begin{adjustbox}{valign=c}
        \begin{subfigure}[b]{0.45\textwidth}
            \centering
            \includegraphics[width=\textwidth]{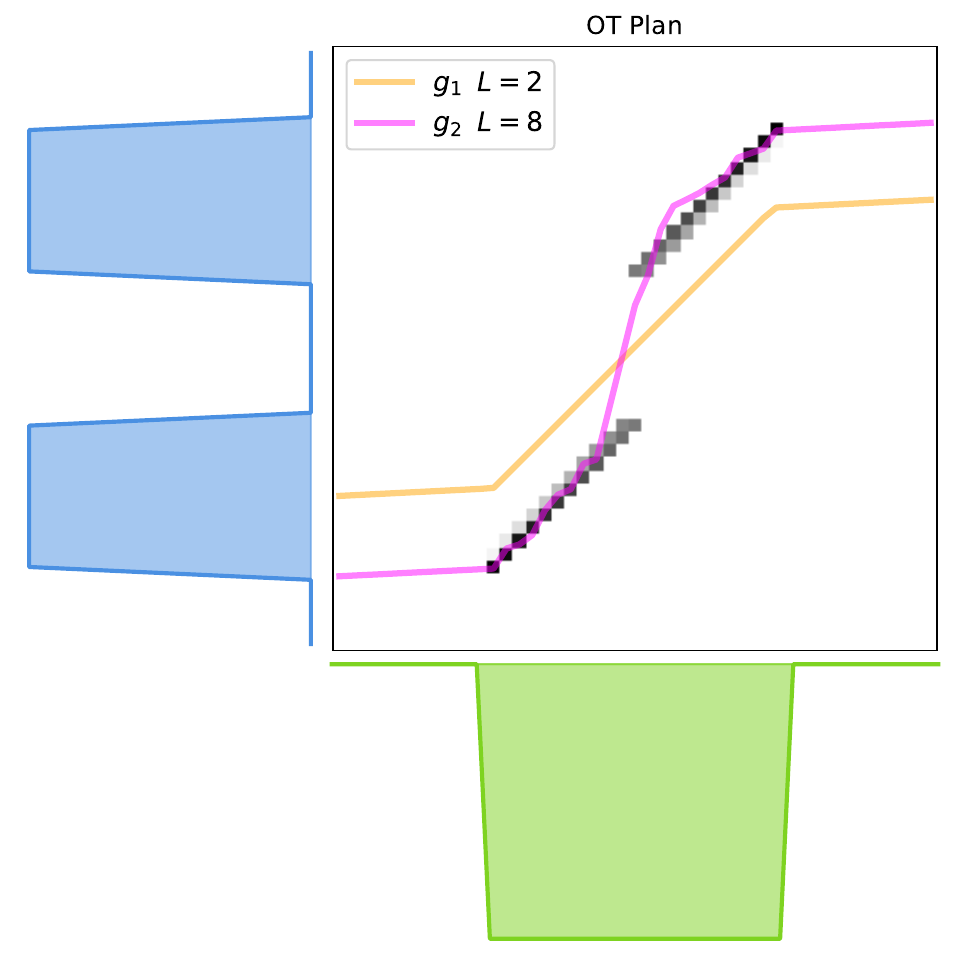}
            \caption{Illustration of map solutions and comparison with the 
            (discontinuous) Optimal Transport coupling.}
        \end{subfigure}
    \end{adjustbox}
    \caption{Illustration of solutions of maps problems (\cref{eqn:CATM}) on a
    toy dataset with a source measure $\mu = \mathcal{U}([-1, 1])$ and a target
    measure $\nu = \frac{1}{2}\mathcal{U}([2, 4]) + \frac{1}{2}\mathcal{U}([6,
    8])$. The two solutions are respectively $L=2$ and $L=8$ Lipschitz.}
    \label{fig:1d_ssnb_map_illu}
\end{figure}

\subsection{Existence of a Solution}

To formulate an existence result, we shall apply the direct method of calculus
of variations, which requires a technical condition on the stability of the
class of functions $G$ with respect to certain limits. To formulate this
condition, we will introduce the notion of closedness of a class of continuous
functions with respect to the compact-open topology. By \cite{kelley2017general}
Chapter 7 Theorem 11, in our setting, this topology is equivalent to the
topology of uniform convergence on compact sets, which allows us to formulate
\cref{def:stable_uniform_limit} in terms of local uniform convergence.

\begin{definition}\label{def:stable_uniform_limit} We say that a set of
    functions $G\subset (\R^d)^\X$ is \blue{\textbf{closed for the compact-open
    topology}} if there exists a sequence $(\K_m)$ of compact sets of $\X$
    verifying $\cup_m \K_m = \X$ such that: \\
  for any sequence $(g_n)_{n\in \N}\in G^\N$ such that for all $m$,
    $(g_n|_{\K_m})_{n\in \N}$ converges uniformly towards a function $g_{\K_m}:
    \K_m \longrightarrow \R^d$, there exists $g\in G$ such that $g|_{\K_m} =
    g_{\K_m}$ for all $m$.
\end{definition}

One can understand this condition as a form of "local uniform closedness" of the
class $G$.

\begin{theorem}\label{thm:existence_G_c} Let $c: \R^d \times \R^d
    \longrightarrow \R_+$ be a \blue{lower semi-continuous} cost function,
    \blue{$\mu\in \mathcal{P}(\X)$ be} a probability measure  on a locally
    compact Polish space $(\X, d_\X)$, and $\nu \in \mathcal{P}(\R^d)$.
    \blue{Assume that}
    \begin{itemize}
        \blue{\item[i) ] \textbf{(Coercive cost)} There exists $\eta: \R_+
        \longrightarrow \R_+$ non-decreasing and such that
        $\eta(t)\xrightarrow[t\longrightarrow +\infty]{} +\infty$,  and $\alpha
        \in \R$ such that $\forall y, y' \in \R^d,\; c(y,y')\geq \alpha +
        \eta(\|y-y'\|_2)$ and $\int \eta(\|\cdot\|_2)\dd\nu < +\infty$;}
        \blue{\item[ii) ] \textbf{(Lipschitzness and Closedness of $G$)} $G$ is
        a subset of the space of $\gradlip$-Lipschitz functions from $\X$ to
        $\R^d$, that is closed for the compact-open topology (see
        \cref{def:stable_uniform_limit});}
        \item[iii)] \textbf{(Problem finiteness)} There exists $g\in G$ such
        that $\T_c(g\#\mu, \nu) < +\infty$.
    \end{itemize}
    Then the problem $\underset{g\in G}{\argmin}\ \T_c(g\#\mu, \nu)$ has a
    solution.
\end{theorem}
\begin{proof}
    \step{1}{Defining a minimising sequence.}
    
    We introduce the notation $J(g) := \T_c(g\#\mu, \nu)$ for convenience, and
    $J^*$ the problem value, which is finite by Assumption iii). Consider a
    minimising sequence $(g_n)_{n\in \N} \in G^\N$ such that 
    $$\forall n \in \N,\; J(g_n) \leq J^* + 2^{-n}.$$

    \blue{\step{2}{Bounding $g_n$.}
      
    First, we fix $n\in \N$, and take $a \in \X$ in the support of $\mu$ and
    $r>0$, then set $A := B_{d_\X}(a, r)$ the ball of centre $a$ and radius $r$
    for the distance $d_\X$, so that $\mu(A) > 0$. The transport problem has a
    solution since $c$ is lower semi-continuous (\cite{santambrogio2015optimal}
    Theorem 1.7). We introduce $\pi_n^* \in \Pi(\mu, \nu)$ optimal for the OT
    cost $\T_c(g_n\#\mu, \nu)$. We lower-bound:
    \begin{align*}
        J(g_n) &\geq \int_{A\times\R^d}c(g_n(x), y)\dd\pi_n^*(x,y) \geq 
        \int_{A\times\R^d}\eta(\|g_n(x)-y\|_2)\dd\pi_n^*(x,y) + \alpha\mu(A).
    \end{align*}
    To separate variables, we will use an elementary inequality: let $z, w\in
    \R^d$, the triangle inequality yields $\|z\|_2\leq 2\max(\|w-z\|_2,
    \|w\|_2)$, applying the non-decreasing and non-negative function $\eta$
    provides $\eta(\|z\|_2/2) \leq \max(\eta(\|w-z\|_2), \eta(\|w\|_2)) \leq
    \eta(\|w-z\|_2) + \eta(\|w\|_2)$. Finally, we have
    \begin{equation}\label{eqn:existence_phi_triangle_ineq}
        \forall w,z \in \R^d,\; \eta(\|w-z\|_2)\geq \eta(\|z\|_2/2) 
        - \eta(\|w\|_2).
    \end{equation}
    By assumption, we remind that $\int \eta(\|\cdot\|_2) \dd\nu < +\infty$, and
    resume lower-bounding using \cref{eqn:existence_phi_triangle_ineq} with $w
    := y$ and $z := g_n(x)$:
    $$J(g_n) \geq \int_A
    \eta\left(\frac{\|g_n(x)\|_2}{2}\right)\dd\mu(x) - \mu(A)\int
    \eta(\|\cdot\|_2) \dd\nu + \alpha\mu(A). $$ Let $x \in A$, we apply again
    \cref{eqn:existence_phi_triangle_ineq} with $w := (g_n(a)-g_n(x))/2$ and $z
    := g_n(a)/2$:
    $$\eta(\|g_n(x)\|_2/2) \geq \eta(\|g_n(a)\|_2/4) -
    \eta(\|g_n(a)-g_n(x)\|_2/2) \geq \eta(\|g_n(a)\|_2/4) - \eta(\gradlip r /
    2),$$ where the second inequality comes from the fact that $g_n$ is
    $\gradlip$-Lipschitz, $d_\X(x, a) \leq r$ and that $\eta$ is non-decreasing.
    Gathering our inequalities leads to the following lower-bound:
    $$J^* + 1 \geq J(g_n) \geq \mu(A)\left(\eta(\|g_n(a)\|_2/4) - \eta(\gradlip
    r / 2)-\int \eta(\|\cdot\|_2)\dd\nu + \alpha\right). $$ This implies that
    there exists $M > 0$ independent of $n$ such that $\|g_n(a)\|_2 \leq M$.
    (Since by coercivity of $\eta$, the right-hand side of the equation above
    would tend to $+\infty$ if $\|g_n(a)\|_2 \xrightarrow[n\longrightarrow +
    \infty]{} + \infty$). }

    \blue{\step{3}{Applying Arzelà-Ascoli's Theorem.}

    For $n\in \N$, we use the upper-bound from Step 2 and the fact that each
    $g_n$ is $\gradlip$-Lipschitz:
    $$\forall x \in \X,\; \|g_n(x)\|_2 \leq M + \gradlip d_\X(x, a),$$ which
    shows that $\blue{\forall x \in \X},\; \{g_n(x),\; n\in \N\}$ has compact
    closure in $\R^d$. The sequence $(g_n)$ is equi-Lipschitz and thus
    equi-continuous, and is closed for the compact-open topology
    (\cref{def:stable_uniform_limit}) by assumption. By Arzelà-Ascoli's theorem
    (as stated in \cite{kelley2017general}, Chapter 7, Theorem 17), we can
    choose $\beta : \N \longrightarrow \N$ an extraction such that $g_{\beta(n)}
    \xrightarrow[n\longrightarrow +\infty]{} g$ locally uniformly on $\X$, for a
    certain function $g \in G$. }

    \step{4}{Showing that the limit $g$ is optimal.}

    \blue{First, the sequence $(g_{\beta(n)}\#\mu)$ converges weakly towards the
    probability measure $g\#\mu$: take a continuous and compactly supported test
    function $\phi: \R^d \longrightarrow \R$, the dominated convergence theorem
    shows that 
    $$\int_\X \phi \circ g_{\beta(n)}\dd\mu
    \xrightarrow[n\longrightarrow+\infty]{} \int_\X \phi \circ g\dd\mu,$$ where
    convergence of the integrands is ensured by the point-wise convergence of
    $(g_{\beta(n)})$, and domination by $\|\phi\|_\infty$ suffices.} Since $c$
    is lower semi-continuous, the OT cost is itself lower semi-continuous for
    the weak convergence of measures \blue{(see \cite{ambrosio2021lectures}
    Theorem 2.6)}, we obtain the following inequality:
    $$\underset{n\longrightarrow+\infty}{\liminf}J(g_{\beta(n)}) \geq J(g),
    $$ 
    where $J$ was introduced in Step 1, where we also chose $g_n$ such as 
    $J(g_n) \leq J^* + 2^{-n}$, thus we conclude $J^* \geq J(g)$, hence $g$ 
    is optimal.

\end{proof}

\cref{thm:existence_G_c} can be extended to the case where the
regularity of the functions of $G$ is only assumed \textit{on a
  partition} of $\X$. \blue{Note that to avoid pathological ambiguity and unnecessary
complications, we will consider partitions whose borders have no mass for $\mu$,
such that the problem objective can be split according to the partition.}

\begin{theorem}\label{thm:existence_G_c_parts}
    Let $c: \R^d \times \R^d
    \longrightarrow \R_+$ be a continuous cost function, a probability measure
    $\mu\in \mathcal{P}(\X)$ on a locally compact Polish space $(\X, d_\X)$, and
    $\nu \in \mathcal{P}(\R^d)$. Consider $(E_i)_{\llbracket 1, K \rrbracket}$ 
    a partition of $\X$ such that for every $i \in \llbracket 1, K \rrbracket,\; \mu(\partial E_i) = 0$.
    Under the same conditions as Theorem~\ref{thm:existence_G_c}, and
    replacing assumption ii) by 
    \begin{itemize}
        \item[ii')] The class of functions $G\subset (\R^d)^\X$ is of the form
        $$G = \left\{g: \X \longrightarrow \R^d\ |\ \forall i \in \llbracket 1, K \rrbracket,\; g|_{\interior{E_i}} = g_i,\; g_i\in G_i\right\}, $$
        where for every $i \in \llbracket 1, K \rrbracket$, the set of functions
        $G_i \subset (\R^d)^{\interior{E_i}}$ is a subset of the space of
        $\gradlip$-Lipschitz functions from $\interior {E_i}$ to $\R^d$, that is
        \blue{closed for the compact-open topology} (see
        \cref{def:stable_uniform_limit}),
    \end{itemize}
    then the problem $\underset{g\in G}{\argmin}\ \T_c(g\#\mu, \nu)$ has a
    solution.
\end{theorem}

\begin{proof}
    We shall follow closely the proof of \cref{thm:existence_G_c}, and point out
    the technical differences. We introduce a minimising sequence exactly
    identically to Step 1. The computations from Step 2 can be done verbatim,
    choosing instead $A_i\subset\interior{E_i}$, and concluding $\|g_n(a_i)\|_2
    \leq M_i$ for a fixed $a_i \in A_i$.

    Step 3 is then done separately on each $\interior{E_i}$, yielding
    extractions $(\beta_i)$ such that each $g_{\beta_i(n)}$ converges locally
    uniformly on $\interior{E_i}$ towards a function $g_i \in G_i$. Considering
    the extraction $\beta := \beta_1 \circ \cdots \circ \beta_K$, we have for
    all $i\in \llbracket 1, K \rrbracket$ the uniform convergence of
    $(g_{\beta(n)})$ towards $g\in G$ on all compact sets of $\interior{E_i}$. 

    Finally, Step 4 is done likewise to \cref{thm:existence_G_c}, with the
    technicality that since $\mu(\partial E_i) = 0$, the pointwise convergence
    of $(g_{\beta(n)})$ towards $g$ at each point of $\interior{E_i}$ suffices
    to show that $g_{\beta(n)}(x) \xrightarrow[n\longrightarrow+\infty]{} g(x)$
    for $\mu$-almost-every $x\in \X$, which yields the convergence in law
    $g_{\beta(n)}\#\mu \xrightarrow[n\longrightarrow+\infty]{w}g\#\mu$. The rest
    follows verbatim. 
\end{proof}

In \cref{remark:generalise_regularisation,remark:generalise_target_space}, we
present some natural extensions of
\cref{thm:existence_G_c,thm:existence_G_c_parts}, which we kept separate for
legibility.

\begin{remark}\label{remark:generalise_regularisation} The existence results of
    \cref{thm:existence_G_c,thm:existence_G_c_parts} also hold if the objective
    functional is changed into a regularised version
    $$J(g) = \T_c(g\#\mu, \nu) + R(g), $$ where $R: G \longrightarrow
    \R_+\cup\{+\infty\}$ is lower semi-continuous with respect to uniform local
    convergence. One also has to assume that there still exists $g\in G$ such
    that the new cost $J$ is finite. The proofs can be written almost
    identically: in Step 1, it suffices to lower-bound $R(g_n) \geq 0$, and in
    Step 4, one obtains $\liminf\ J(g_{\beta(n)}) \geq J(g)$ thanks to
    the lower semi-continuity of $R$.
\end{remark}

\begin{remark}\label{remark:generalise_target_space} Condition i) on $c$ can be
	generalised to the case where the target space $\R^d$ is instead a Polish
	space $\Y$ verifying the Heine-Borel property (i.e. all closed and bounded
	sets are compact), in which case Condition i) can be replaced with the
	condition that $c(\cdot, y_0)$ be \textbf{proper}, which is to say that its
	preimage by any compact set $S \subset \R_+$ is a compact set of $\Y$. This
	property would be used in Step 2 to show that $g_n(a)\in C$ for some compact
	set $C\subset\Y$ independent of $n$, then in Step 3, we would use the
	Lipschitz property of $g_n$ and the triangle inequality on $d_\Y$ to show
	that $\forall x\in \K,\; g_n(x) \in \oll{B}_\Y(y_0, \gradlip r + d_\Y(y_0,
	C))$, for a compact set $\K\subset \X$ of diameter $r$ and $y_0\in \Y$. This
	would show that for each $x\in \K$, the set $\{g_n(x)\}_{n\in \N}$ is
	pre-compact in $\Y$, and allow one to apply Arzelà-Ascoli likewise.
\end{remark}

A natural context for Optimal Transport is the case where the ground cost is of
the from $c(x,y) = \|x-y\|^p$ for some norm $\|\cdot\|$ on $\R^d$ and $p\geq 1.$
In \cref{prop:costs_verifying_assumptions}, we show that such costs verify the
assumptions to our existence theorems.

\begin{prop}\label{prop:costs_verifying_assumptions} Cost functions of the form
    $c(x,y) := \|x-y\|^p$, where \blue{$p> 0$} and $\|\cdot\|$ is a norm on
    $\R^d$ satisfy Assumption i) of
    \cref{thm:existence_G_c,thm:existence_G_c_parts}, as long as $\nu \in
    \mathcal{P}_{p}(\R^d)$.
\end{prop}
\begin{proof}
    \blue{Take $\eta := t \longmapsto (Kt)^p$, where $K>0$ is provided by the
    norm equivalence inequality $\|\cdot\|\geq K\|\cdot\|_2$.}
\end{proof}

\subsection{Function Class Example: Gradients of Convex
Functions}\label{sec:class_grad_convex}

An interesting class of functions $G$ to optimise over is the set of
$\gradlip$-Lipschitz functions that are gradients of ($\strcvx$-strongly) convex
functions. Indeed, this can be seen as a regularising assumption, and was
studied in \cite{paty2020regularity} for the cost $c(x,y) = \|x-y\|_2^2$. We
shall see in \cref{prop:grad_cvx_verify_assumptions} that classes of such
functions on \textit{arc-connected} partitions verify the conditions of our
existence result \cref{thm:existence_G_c_parts}. In particular,
\cite{paty2020regularity} Definition 1 (which states existence, with a
simplified proof due to lack of space) is a consequence of
\cref{thm:existence_G_c_parts}. Before this result, we will present a technical
lemma on arc-connectedness. \blue{In this paper, we will say that a set
$A\subset \R^d$ is \textit{arc-connected} if any pair of points of $A$ can be
joined by a Lipschitz curve contained in $A$.}

\begin{lemma}\label{lemma:union_arc_connected} Let $\O$ be an arc-connected open
	set of $\R^d$. There exists $(C_k)_{k\in \N}$ a sequence of arc-connected
	compact sets such that $\forall k \in \N,\; C_k \subset C_{k+1}$ and
	$\Reu{k\in \N}{}C_k = \O$.
\end{lemma}
\begin{proof}
    \blue{Consider the collection $(\oll{B}(q, r_q))_{q \in \O\cap\Q^d}$ where
    for each $q\in \O\cap\Q^d$, we take $r_q > 0$ such that $\oll{B}(q, r_q)
    \subset \O$. Using a bijection between $\N$ and $\O\cap\Q^d$, we can
    introduce sequences $(q_k) \in (\O\cap\Q^d)^\N$ and $(r_k) \in (0,
    +\infty)^\N$ such that the sequence of the $A_k := \oll{B}(q_k, r_k)$
    enumerates the previous collection. The sequence $(A_k)$ is made of compact
    arc-connected sets and verifies $\O = \cup_kA_k$. We can now defined
    recursively the sequence $(C_k)$ by $C_0 := A_0$ and $C_{k+1} := C_k \cup
    A_{k+1} \cup w_k([0, 1])$, where for $k\in \N$, $w_k: [0, 1] \longrightarrow
    \O$ is a Lipschitz curve between $q_k$ and $q_{k+1}$ contained in $\O$
    (which exists by assumption on $\O$). By induction, the sequence $(C_k)$
    verifies the desired properties.}
\end{proof}

\blue{We now have the technical tools to prove that $\gradlip$-Lipschitz
functions that are gradients of ($\strcvx$-strongly) convex functions verifies
the local convergence stability assumption of \cref{thm:existence_G_c_parts} on
partitions of $\R^d$.}

\begin{prop}\label{prop:grad_cvx_verify_assumptions} Consider $\X := \R^d$, and
    a partition $\mathcal{E} := (E_i)_{\llbracket 1, K \rrbracket}$, where each
    $\interior{E_i}$ is arc-connected. Let $0\leq \strcvx\leq \gradlip$, the set
    of functions
    $$\Funs := \left\{g: \R^d \longrightarrow \R^d\ |\ \forall i \in \llbracket
    1, K \rrbracket,\; g|_{\interior{E_i}}\ \gradlip-\text{Lipschitz};\;
    g|_{\interior{E_i}} = \nabla \pot_i,\; \pot_i \in
    \mathcal{C}^1(\interior{E_i}, \R),\; \pot_i\ \strcvx-\text{strongly\ convex}
    \right\} $$ verifies Assumption \blue{ii')} of
    \cref{thm:existence_G_c_parts}.
\end{prop}
\begin{proof}
    Let $i \in \llbracket 1, K \rrbracket$, and define for notational
    convenience $\U:=\interior{E_i}$. We want to show that the set of functions
    $$G := \left\{g: \U \longrightarrow \R^d\ \gradlip-\text{Lipschitz}\ |\ g =
    \nabla\pot,\; \pot \in \mathcal{C}^1(\U, \R),\; \pot\
    \strcvx-\text{strongly\ convex}\right\} $$ is \blue{closed for the
    compact-open topology} (\cref{def:stable_uniform_limit}). By
    \cref{lemma:union_arc_connected}, since $\U$ is open and arc-connected, we
    can choose an increasing sequence of arc-connected compact sets $\K_m
    \subset \U$ such that $\cup_m\K_m = \U$. We fix $a\in \K_0$.

    Take a sequence $(g_n)_{n\in \N} \in G^\N$ such that for each $m\in \N$,
    $g_n|_{\K_m}$ converges uniformly to some function $h_m \in
    \mathcal{C}^0(\K_m, \R^d)$. We will show that there exists $g\in G$ that
    coincides with $h_m$ on each $\K_m$. Regarding the Lipschitz constraint, by
    point-wise convergence, each function $h_m$ is $\gradlip$-Lipschitz. 
    
    For any $n\in \N$, since $g_n\in G$, we can introduce an $\strcvx$-strongly
    convex function $\pot_n \in \mathcal{C}^1(\U, \R)$ such that $g_n = \nabla
    \pot_n$. Since $\pot_n$ can be chosen up to an additive constant, we can
    assume $\pot_n(a)=0$. We study the point-wise convergence of $(\pot_n)$ on
    $\K_m$ for $m\in \N$ fixed, so we fix $x\in \K_m$. Since $\K_m$ is
    arc-connected, we can choose $w: [0, 1] \longrightarrow \K_m$ \blue{a
    Lipschitz curve} such that $w(0)=a$ and $w(1)=x$. Noticing that \blue{for
    almost-every $t\in [0, 1],$} $\frac{\dd}{\dd t} \pot_n (w(t))=\langle \nabla
    \pot_n(w(t)), \dot w(t) \rangle$ and using $\pot_n(a) = 0$, we write
    \blue{(by absolute continuity of $\pot_n \circ w$)}:
	$$\pot_n(x) = \Int{0}{1}\langle\nabla \pot_n(w(t)), \dot w(t)\rangle \dd t
	\xrightarrow[n \longrightarrow +\infty]{} \Int{0}{1}\langle h_m(w(t)), \dot
	w(t)\rangle \dd t =: \psi_m(x),$$ where the convergence is obtained \blue{by
	the dominated convergence theorem.}

    Our objective is now to prove that $\psi_m$ is $\mathcal{C}^1$-smooth on
	$\interior{\K}_m$, and that $\nabla \psi_m = h_m$. Let $x \in
	\interior{\K}_m$, $v \in \R^d$ and $\delta > 0$ such that $\forall t \in
	[-\delta, \delta],\; x+tv \in \interior{\K}_m$. For $n \in \N$ and $t \in
	[-\delta, \delta]$, let $f_n(t) := \pot_n(x+tv)$. We have shown that the
	sequence $(f_n)$ converges pointwise to $f :=t \longmapsto \psi_m(x+tv)$.
	Furthermore, by convergence of $(g_n)$, the derivative sequence $f_n' = t
	\longmapsto \langle \nabla \pot_n(x+tv), v \rangle$ converges uniformly on
	$[-\delta, \delta]$ to $t \longmapsto \langle h_m(x+tv), v \rangle$. A
	standard calculus theorem then shows that $f$ is differentiable on
	$(-\delta, \delta)$, with $f'(t) = \frac{\dd} {\dd t} \langle h_m(x+tv), v
	\rangle$. In particular, by setting $t=0$ we have shown that the directional
	derivative $D_v\psi_m(x)$ exists and has the value $\langle h_m(x), v
	\rangle$. Since $h_m$ is continuous (we saw that it is Lipschitz), this
	shows that $\psi_m$ is of class $\mathcal{C}^1$, with $\nabla \psi_m = g_m$
	on $\interior{\K}_m$.

    For $x \in \U$, letting $m := \min\lbrace m \in \N : x \in \K_m\rbrace$, we
	define $\psi(x) := \psi_m(x)$, which is well-defined since $x \in
	\mathcal{K}_m$. For $m < m'$, since $\K_m \subset \K_{m'}$, we have
	$\psi_{m'}|_{\K_m} = \psi_m$, as a consequence, for any $m\in \N,\;
	\psi|_{\K_m} = \psi_m$ without ambiguity. The previous result implies in
	particular that $\psi$ is of class $\mathcal{C}^1$ on each
	$\interior{\K}_m$, and thus everywhere on $\U$. We define $g: \U
	\longrightarrow \R^d$ similarly, with the same property $g|_{\K_m} = h_m$.
	With this construction, on each $\interior{\K}_m$, one has $g = g_m = \nabla
	\psi_m = \nabla \psi$. As a result, we have $g = \nabla \psi$ on all of
	$\U$. Since each $g_m$ is $\gradlip$-Lipschitz, it follows that $g$ is
	$\gradlip$-Lipschitz on $\U$.

    To see that $g\in G$, it only remains to show that $\psi$ is
    $\strcvx$-strongly convex, which is a consequence of the fact that it
    is everywhere a point-wise limit of a $\psi_m$, which is itself
    $\strcvx$-strongly convex.
\end{proof}

\subsection{Function Class Example: Neural Networks}

Another natural idea is to consider classes $G$ of parametrised functions, in
particular Neural Networks (NNs) with Lipschitz activation functions. We will
consider a relatively general expression for NNs borrowed from
\cite{tanguy2023convergence}. We consider a class $\GNN$ of functions $g_\theta
= h_N(\theta, \cdot): \R^k \longrightarrow \R^d$ for a parameter vector
$\theta\in \Theta$, where $\Theta \subset \R^p$ is a compact set, and where
$h_N$ is the $N$-th layer of a recursive NN structure defined by
\begin{equation}\label{eqn:def_NN}
	\begin{split} 
		&h_0(\theta, x) = x, \quad \forall n \in \llbracket 1,
		N \rrbracket,\; h_n =
		\app{\R^p\times\R^k}{\R^{d_n}}{(\theta,x)}{a_n\left(\Sum{i=0}{n-1}
		A_{n,i}(\theta)h_i(\theta,x) + b_n \theta \right)}, \\
		&N \in \N,\; d_0 = k,\; d_N=d,\; \forall n \in \llbracket 1, N \rrbracket,\; d_n \in \N^*,\\
		&a_n: \R^{d_n} \longrightarrow \R^{d_n}\ \text{Lipschitz},\; b_n \in \mathcal{L}(\R^p, \R^{d_n}),\; \forall i \in \llbracket 0, n-1 \rrbracket,\; A_{n,i} \in \mathcal{L}(\R^p, \R^{d_n\times d_i}),
	\end{split}
\end{equation}
where $\mathcal{L}(A, B)$ is the space of linear maps from $A$ to $B$. The terms
$A_{n,i}$ and $b_n$ correspond to the weights matrices and biases respectively,
and we allow the use of the entire parameter vector $\theta\in
\Theta\subset\R^p$ at each layer for generality. The summation over the previous
layers allows the inclusion of ``skip-connections" in the architecture. Thanks
to the assumption that the parameters lie in a compact set, we will show that
the class $\GNN$ verifies the conditions of our existence theorem
\cref{thm:existence_G_c}.

\begin{prop}\label{prop:GNN_verify_assumptions} Let $\Theta\subset \R^p$ be a
	compact set and $\GNN$ the class of functions $\R^p \longrightarrow \R^d$ of
	the form $g_\theta = h_N(\theta, \cdot)$, with $\theta\in\Theta$ and $h_N$
	as in \cref{eqn:def_NN}. Then $\GNN$ verifies Assumption ii) of
	\cref{thm:existence_G_c}.
\end{prop}
\begin{proof}
	An immediate induction over the layers shows that for $g_\theta\in \GNN$,
	there exists a constant $L>0$ independent of $\theta$ such that $g_\theta$
	is $L$-Lipschitz on $\R^k$. \\
	Concerning \blue{closedness for the compact-open topology}
	(\cref{def:stable_uniform_limit}), we will show the following stronger
	property: if $(g_m)\in (\GNN)^\N$ converges pointwise towards a function
	$f:\R^k \longrightarrow \R^d$, then there exists $\theta\in \Theta$ such
	that $f = g_\theta$. For $m\in \N$, we can write $g_m = g_{u_m}$ for $u_m
	\in \Theta$. Since the sequence $(u_m)$ lies in the compact set $\Theta$,
	there exists a converging subsequence $(u_{\alpha(m)})$ which converges
	towards $\theta\in \Theta$. Let $x\in \R^k$, we have the convergence
	$g_{u_m}(x) \longrightarrow f(x)$. By induction over the layers, the
	function $v\longmapsto g_v(x)$ is continuous, thus $g_{u_{\alpha(m)}}(x)
	\longrightarrow g_{\theta}(x)$. By uniqueness of the limit,
	$f(x)=g_\theta(x)$, and since $x\in \R^k$ was chosen arbitrarily, we
	conclude $f \in G$.
\end{proof}

\begin{remark}
	For simplicity, we presented NNs taking $x\in\R^k$ as input, yet the theory
	holds if $\X$ is a locally compact Polish space, just as in
	\cref{thm:existence_G_c}. For instance, one could take a Riemannian
	manifold.
\end{remark}

\subsection{On the Necessity of the Lipschitz Constraint for
Existence}\label{sec:no_lip_no_existence}

Beyond the theoretical usefulness of the constraint that $g$ be
$\gradlip$-Lipschitz, this constraint may add substantial difficulty to the
numerical implementation (see \cref{sec:discrete}). As a result, one could
consider the map problem \cref{eqn:CATM} without the Lipschitz assumption on
$G$. Unfortunately, this variant has no solution in general. We illustrate this
in the light of the class of functions $\Funs$ introduced in
\cref{prop:grad_cvx_verify_assumptions}, in the 1D case and consider $G$ the
cone of continuous non-decreasing functions, yielding the problem:
\begin{equation}\label{eqn:SSNB_only_convex}
	\underset{g \in \mathcal{C}^0(\R),\ \text{non-decreasing}}{\argmin}\ \W_2^2(g\#\mu, \nu),
\end{equation}
where we choose the specific measures $\mu := \mathcal{U}([-1, 1])$ and $\nu :=
\frac{1}{2}\mathcal{U}([-2, -1]) + \frac{1}{2}\mathcal{U}([1, 2])$. In this
setting, no continuous function $g: \R \longrightarrow \R$ can satisfy $g\#\mu =
\nu$. Indeed, suppose that such a continuous function $g$ were to exist. On the
one hand, since $g$ is continuous, $\supp(g\#\mu) = g(\supp(\mu)) = g([-1, 1])$.
On the other hand, by assumption $\supp(g\#\mu) = \supp(\nu) = [-2, -1]\cup [1,
2]$. However, since $g$ is continuous and $[-1, 1]$ is connected, $g([-1, 1])$
is also connected, thus $[-2, -1]\cup [1, 2]$ is connected, which is a
contradiction.

We now consider a specific function $g$ which satisfies $g\#\mu = \nu$:
\begin{equation}\label{eqn:CE_existence_g}
    g:= \app{\R}{\R}{x}{\left\lbrace\begin{array}{c} x-1\ \text{if}\ x<0 \\
		0\ \text{if}\ x=0 \\
		x+1\ \text{if}\ x>0 \\
	\end{array} \right.},
\end{equation}
note that the value at 0 can be chosen arbitrarily. This function is not
continuous, so we approach it by functions $g_\varepsilon$, with $\varepsilon
\in (0,1)$, which are continuous and non-decreasing:
\begin{equation}\label{eqn:CE_existence_g_eps}
    g_\varepsilon:= \app{\R}{\R}{x}{\left\lbrace\begin{array}{c} x-1\ \text{if}\
		x\leq-\varepsilon \\
		\frac{1+\varepsilon}{\varepsilon}x\ \text{if}\ x\in [-\varepsilon,
		\varepsilon] \\
		x+1\ \text{if}\ x\geq\varepsilon \\
	\end{array} \right.}.
\end{equation}    
A straightforward computation yields:
\begin{equation}\label{eqn:CE_existence_g_eps_img}
    g_\varepsilon\#\mu = \cfrac{1-\varepsilon}{2}\mathcal{U}([-2, -1-\varepsilon])
+ \varepsilon\mathcal{U}([-1-\varepsilon, 1 + \varepsilon]) +
\cfrac{1-\varepsilon}{2}\mathcal{U}([1+\varepsilon, 2]),
\end{equation}
which we illustrate in \cref{fig:CE_existence}.
\begin{figure}[ht]
    \centering
    \begin{subfigure}[b]{0.45\textwidth}
        \centering
        \includegraphics[width=\textwidth]{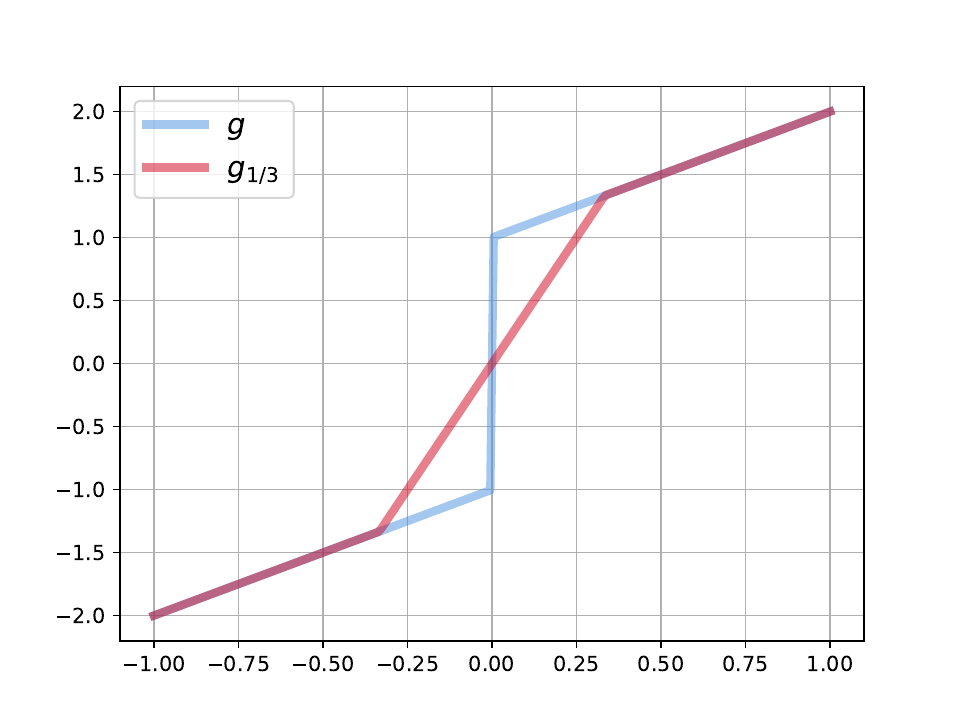}
        \caption{Illustration of the maps $g$ from \cref{eqn:CE_existence_g} and
        $g_\varepsilon$ from \cref{eqn:CE_existence_g_eps} with
        $\varepsilon=1/3$.}
    \end{subfigure}
    \hfill
    \begin{subfigure}[b]{0.45\textwidth}
        \centering
        \includegraphics[width=\textwidth]{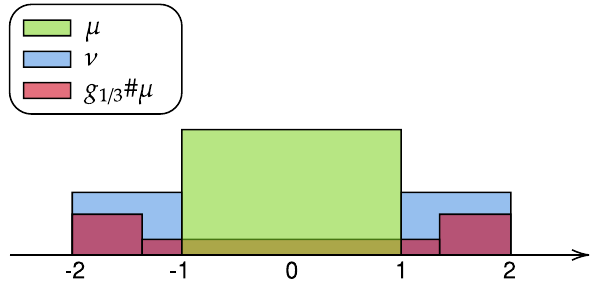}
        \caption{Illustration of the image measure $g_{1/3}\#\mu$ with $\mu =
        \mathcal{U}([-1, 1])$ and $g_\varepsilon$ from
        \cref{eqn:CE_existence_g_eps}.}
    \end{subfigure}
    \caption{Illustration of the counter-example to existence.}
    \label{fig:CE_existence}
\end{figure}
It follows that $g_\varepsilon\#\mu$ converges weakly towards $\nu$ as
$\varepsilon \longrightarrow 0$. As a result, since the measures are compactly
supported, $\W_2^2(g_\varepsilon\#\mu, \nu) \xrightarrow[\varepsilon
\longrightarrow 0]{} 0$, thus the value of Problem \cref{eqn:SSNB_only_convex}
is 0.

To conclude, if Problem \cref{eqn:SSNB_only_convex} had a solution $g$, then it
would be continuous and verify $\W_2^2(g\#\mu, \nu) = 0$ (since the problem
value is 0), thus $g\#\mu = \nu$, which is impossible by the connectivity
argument. Therefore, the problem defined in \cref{eqn:SSNB_only_convex} does not
have a solution.

\subsection{Discussion on Uniqueness}\label{sec:uniqueness}

A natural question is the uniqueness of a solution of the problem
\begin{equation*}
    \underset{g\in G}{\argmin}\ \mathcal{T}_c(g\#\mu, \nu),
\end{equation*}
in the case where the measures, the cost and the class $G$ satisfy the
conditions of \cref{thm:existence_G_c}, guaranteeing existence. A first negative
answer concerns the simple case where $\mu, \nu$ are discrete and at least
two-dimensional. For instance, consider
$$\mu := \frac{1}{2}(\delta_{(-1, 0)} + \delta_{(1, 0)}),\quad \nu :=
\frac{1}{2}(\delta_{(0, -1)} + \delta_{(0, 1)}). $$ Then there are two distinct
maps $g_1, g_2$ both verifying $g_i\#\mu=\nu$, which are characterised in
$L^2(\mu)$ by their values on the two points $(\pm 1, 0)$.
$$g_1((-1, 0)) = (0, -1),\; g_1((1, 0)) = (0, 1),\quad g_2((-1, 0)) = (0, 1),\;
g_2((1, 0)) = (0, -1),$$ as we illustrate in \cref{fig:two_maps_discrete}.
\begin{figure}[ht]
    \center
    \includegraphics[width=0.3\linewidth]{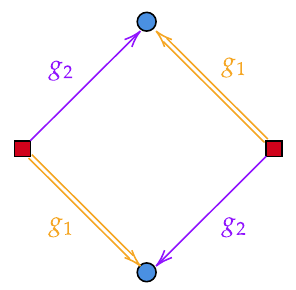}
    \caption{A simple case with two transport maps between 2-point discrete measures in $\R^2$.}
    \label{fig:two_maps_discrete}
\end{figure}
The previous example illustrates a potential issue for uniqueness, which is the
multiplicity of the set $\{g\in G\ |\ g\#\mu = \nu\}$. Another simple
counter-example to uniqueness which stems from this property is for
$\mu=\nu=\mathcal{N}(0, I)$ the standard $d$-variate Gaussian distribution. In
this case, any rotation $R$ verifies $R\#\mathcal{N}(0, I) = \mathcal{N}(0, I)$.

More generally, Brenier's polar factorisation theorem \cite{brenier1991polar}
sheds a light on our invariance issue. We present the theorem below for
completeness, see also \cite{santambrogio2015optimal} Section 1.7.2.
\begin{theorem}[Brenier's Polar Factorisation
    \cite{brenier1991polar}]\label{thm:brenier_polar_facto} Let $\K\subset \R^d$
    be a compact set, and $g: \K \longrightarrow \R^d$. Consider $\U_\K$ the
    probability measure that is the uniform distribution on $\K$, suppose that
    $g\#\U_\K\ll \Leb$, then there exists a unique ($\Leb$-almost-everywhere)
    decomposition $g= (\nabla \pot) \circ s$ such that:
    \begin{itemize}
        \item $\pot: \K \longrightarrow \R^d$ is convex;
        \item $s: \K \longrightarrow \K$ is measure-preserving, which is to say
        that $s\#\U_\K = U_\K$.
    \end{itemize}
\end{theorem}
To fix the ideas, if we consider $\mu = \U_{[0, 1]^d}$, we can fix $g\in G$ and
assume $g\#\U_\K\ll \Leb$ \blue{(see sufficient conditions for this in
\cref{lemma:image_measure_AC} in the Appendix)}, then decompose $g=\nabla\pot
\circ s$. Then any map $h$ of the form $\nabla\pot \circ r$ with $r$ a
measure-preserving map will verify $h\#\U_{[0,1]^d} = g\#\U_{[0,1]^d}$. To avoid
such potential counter-examples, we will focus on the case where $G$ is a subset
of gradients of convex functions. 

We provide a uniqueness result for the $\W_2$ case, under the simplifying
assumption that $\nu = \mu$. Note that if $L < 1$, the identity map does not
belong to $G$, and there does not exist a $g\in G$ such that $g\#\mu = \mu$.
\begin{prop}\label{prop:uniqueness_grad_conv} Suppose that
    $$ G = \left\lbrace g: \R^d \longrightarrow \R^d : g = \nabla \pot\ \Leb
    -\text{a.e.},\; \pot\ \text{convex},\; g\ L-\text{Lipschitz}
    \right\rbrace,$$ and that $\mu\in \mathcal{P}_2(\R^d)$ with $\mu \ll \Leb$.
    Then if $g_0$ and $g_1$ are solutions of the problem
    $$\underset{g\in G}{\argmin}\ \W_2^2(g\#\mu, \mu),$$ then $g_0 = g_1$
    everywhere on $\supp(\mu)$.
\end{prop}
\begin{proof}
    We will show that if $g_0$ and $g_1$ are solutions, then
    $g_0\#\mu=g_1\#\mu$. First, one may write $g_i = \nabla \pot_i$ with
    $\pot_i$ convex (for $i=0, 1$). By \cite{santambrogio2015optimal} Theorem
    1.48, since $\pot_i$ is convex, $g_i$ is the optimal transport map between
    $\mu$ and $\nabla\pot_i \#\mu$. Consider for $t \in [0, 1]$ the
    interpolation $g_t := (1-t)g_0 + tg_1$. Then by definition (see
    \cite{ambrosio2005gradient}, Section 9.2), the curve $(g_t\#\mu)_{t\in [0,
    1]}$ is a\footnote{In this case, since $\mu \ll \Leb$, there is even
    uniqueness of \textit{the} generalised geodesic between $g_0\#\mu$ and
    $g_1\#\mu$, but we do not use that fact.} generalised geodesic between
    $g_0\#\mu$ and $g_1\#\mu$ with respect to the base measure $\mu$. This
    allows us to apply \cite{ambrosio2005gradient} Lemma 9.2.1, specifically
    Equation 9.2.7c, which yields
    $$\forall t \in [0, 1],\; \W_2^2(g_t\#\mu, \mu) \leq (1-t)\W_2^2(g_0\#\mu,
    \mu) + t\W_2^2(g_1\#\mu, \mu) - t(1-t)\W_2^2(g_0\#\mu, g_1\#\mu). $$ The
    curvature of this generalised geodesic will allow us to build a better
    solution if $g_0\#\mu \neq g_1\#\mu$, as we illustrate in
    \cref{fig:unicite_absurde}.
    \begin{figure}[H]
        \centering
        \includegraphics[width=0.35\linewidth]{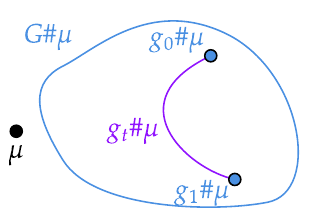}
        \caption{The generalised geodesic based on $\mu$ between $g_0\#\mu$ and $g_1\#\mu$.}
        \label{fig:unicite_absurde}
    \end{figure}
    Taking $t = 1/2$ yields, using the optimality of $g_0$ and $g_1$ and writing
    $v$ for the problem value:
    $$\W_2^2(g_{\frac{1}{2}}\#\mu, \mu) \leq v - \frac{1}{4}\W_2^2(g_0\#\mu,
    g_1\#\mu).$$ Since $G$ is convex, we have $g_{\frac{1}{2}}\in G$, which
    imposes $\W_2^2(g_0\#\mu, g_1\#\mu) = 0$, since $v$ is the optimal problem
    value. We conclude $g_0\#\mu = g_1\#\mu$. However, as stated earlier, by
    \cite{santambrogio2015optimal} Theorem 1.48, $g_i$ is the optimal transport
    map between $\mu$ and $g_i\#\mu$ for $i=0,1$. By uniqueness of the optimal
    transport map in this setting, we conclude $g_0 = g_1$. (The equality holds
    $\mu$-a.e., then since $g_0$ and $g_1$ are assumed Lipschitz, this shows
    equality everywhere on $\supp(\mu)$.)
\end{proof}
\begin{remark}
    One could replace the set $G$ in \cref{prop:uniqueness_grad_conv} by a
    convex subset of $G$, the proof of the result would follow verbatim.
\end{remark}
\begin{remark}
    The problem in \cref{prop:uniqueness_grad_conv} is related to the problem of
    the Wasserstein metric projection, which was studied in \cite{de2016bv} (see
    Section 5), from which the curvature argument in our proof was closely
    inspired. This Wasserstein projection problem was also studied for $\W_p^p$
    in \cite{adve2020nonexpansiveness}. 
\end{remark}
\begin{remark}
    Under some assumptions, it may be possible to find subclasses of gradients
    of convex functions $G$ such that the set $G\#\mu\subset\mathcal{P}_2(\R^d)$
    is geodesically convex (with respect to $\W_2$ geodesics): take $g_0, g_1
    \in G$, assume that $g_0\#\mu\ll\Leb$ (\cref{lemma:image_measure_AC}
    provides a sufficient condition on $g_0$ and $\mu$ for this to be the case).
    Then the $\W_2$ geodesic from $g_0\#\mu$ to $g_1\#\mu$ is
	$$\nu_t := ((1-t)I + tT)\#g\#\mu_0,$$ where $T$ is the optimal transport map
	from $g_0\#\mu$ to $g_1\#\mu$, which is uniquely defined thanks to Brenier's
	Theorem (see \cite{santambrogio2015optimal}, Theorem 1.22 for a possible
	reference without compactness assumptions). Since $(T\circ
	g_0)\#\mu=g_1\#\mu$, under some regularity assumptions, it may be possible
	to show that $T\circ g_0 = g_1$ using the Monge-Ampère equation, then
	$((1-t)I+tT)\circ g_0 = (1-t)g_0 + tg_1 \in G$. In this case, the
	generalised geodesic based on $\mu$ coincides with the $\W_2$ geodesic
	between $g_0\#\mu$ and $g_1\#\mu$. 
	
	Unfortunately, $\rho \longmapsto \W_2^2(\rho, \nu)$ is not convex along
	$\W_2$ geodesics, since it satisfies the opposite inequality
	(\cite{ambrosio2005gradient}, Theorem 7.3.2). As a result, even if we found
	a convex class $G$ of gradients of convex functions such that $G\#\mu$ were
	geodesically convex, curvature arguments would not yield uniqueness
	immediately. Intuition suggests that in some sense, the problem minimises a
	concave function over a convex set, which bodes poorly with uniqueness.	
\end{remark}

In \cref{sec:1D}, we shall study the case $d=1$ and show uniqueness and an
explicit expression for the minimiser of the map problem for non-decreasing
functions $g$ and the squared Euclidean cost. \blue{To conclude this discussion,
even for the favourable case where $\mu\ll \Leb^d$, $G$ is a subset of gradients
of convex functions and $c(x,y) = \|x-y\|_2^2$, we conjecture that uniqueness is
not guaranteed in general for $d\geq 2$. }

\subsection{The Plan Approximation Problem}\label{sec:plan_problem}

In some cases, one may have access to a transport plan between two measures
$\mu,\nu$, which poses the natural question of finding a map that approximates
this transport plan. For instance, one may compute the optimal entropic plan
\cite{cuturi2013sinkhorn}, a Gaussian-Mixture-Model optimal plan
\cite{delon2020wasserstein}, or an optimal transport plan for a cost that does
not verify the twist condition (see \cite{santambrogio2015optimal} Definition
1.16), \blue{or more generally an optimal plan when the Monge problem is not
equivalent to the Kantorovich problem (see
\cite{brenier1991polar,gangbo1996geometry,pratelli2007equality} for some known
equivalence cases).} 

Given a cost $C: (\R^k\times\R^d)\times (\R^k\times\R^d) \longrightarrow \R_+$,
and measures $\mu\in \mathcal{P}(\R^k), \nu \in \mathcal{P}(\R^d)$, we will want
to approximate a plan $\gamma\in\Pi(\mu, \nu)$ by the image measure $(I,
g)\#\mu$, where $I$ denotes the identity map of $\R^k$. We define the
Constrained Approximate Transport Plan problem as:
\begin{equation}\label{eqn:CATP}
	\underset{g \in G}{\argmin}\ \TC((I, g)\#\mu, \gamma).
\end{equation}
Similarly to \cref{eqn:CATM}, the transport cost in \cref{eqn:CATP} can be
re-written using the change-of-variables formula
(\cref{lemma:plans_and_push_forwads}):
\begin{equation}\label{eqn:CATP_CoV}
    \TC((I, g)\#\mu, \gamma) = \underset{\rho\in \Pi(\mu, \gamma)}{\min}\ \int_{\R^k\times(\R^k\times\R^d)}C\left((x, g(x)), (y_1 ,y_2)\right)\dd\rho(x, y_1, y_2).
\end{equation}
To begin with, one may cast \cref{eqn:CATP} as a map problem (\cref{eqn:CATM}),
providing existence automatically under adequate conditions.
\begin{corollary}\label{cor:existence_CATP} Consider the class of functions 
    $$\widetilde{G} := \{\tilde{g} : \R^k \longrightarrow \R^k\times\R^d\ :\
    \tilde{g} = (x,y) \longmapsto (x, g(y)),\; g \in G\},$$ the map problem
    (\cref{eqn:CATM}) is a particular map problem (\cref{eqn:CATP}):
    $$\underset{g \in G}{\min}\ \TC((I, g)\#\mu, \gamma) = \underset{\tilde{g}
    \in \widetilde{G}}{\min}\ \TC(\tilde{g}\#\mu, \gamma),$$ hence existence
    holds by \cref{thm:existence_G_c} if the conditions of the theorem are
    verified by $C, \widetilde{G}$ and the measures $\mu, \gamma$.
\end{corollary}
\begin{remark}
    In the light of \cref{remark:generalise_target_space}, one could replace the
    input space $\R^k$ and the target space $\R^d$ by Polish spaces $\X$ and
    $\Y$ verifying the Heine-Borel property, in which case condition 1) would
    ask for $(x_1, x_2) \longmapsto C((x_1, x_2), (y_1, y_2))$ to be proper.
\end{remark}
We shall see that in certain cases, the two problems \cref{eqn:CATP} and
\cref{eqn:CATM} are in fact equivalent. 

\begin{prop}\label{prop:GSSNB_equivalence} Consider a cost $C$ of the separable
	form $C((x_1, x_2), (y_1, y_2)) = h(c_1(x_1, y_1), c_2(x_2, y_2))$, where
	$h: \R_+ \times \R_+ \longrightarrow \R_+$, $c_1: \R^k\times \R^k
	\longrightarrow \R_+$ and $c_2: \R^d\times\R^d$ are \blue{lower
	semi-}continuous, with $\forall x \in \R^k,\; c_1(x, x) = 0$, and $\forall
	u, v \in \R_+,\; h(u, v)\geq v$ and $h(0, v)=v$. Let $g: \R^k
	\longrightarrow \R^d$ be a measurable function, $\nu \in \mathcal{P}(\R^d)$
	and $\mu\in \mathcal{P}(\R^k)$. Let $\gamma \in \Pi(\mu, \nu)$ be a plan
	between $\mu$ and $\nu$.

    We assume that the value $\TC((I, g)\#\mu, \gamma)$ is finite. We have the
    equality
	$$\mathcal{T}_{c_2}(g\#\mu, \nu) = \TC((I, g)\#\mu, \gamma).$$
\end{prop}

\begin{proof}
	For $\rho \in \Pi(\mu, \gamma)$, let $A(\rho) :=
	\int_{\R^k\times(\R^k\times\R^d)}C((x, g(x)), (y_1, y_2))\dd\rho(x, y_1,
	y_2) <+\infty$, and denote $A^* := \TC((I, g)\#\mu, \gamma)$. Likewise, for
	$\pi \in \Pi(\mu, \nu),$ let $B(\pi) := \Int{\R^k\times\R^d}{}c_2(g(x),
	y)\dd\pi(x, y)$, and $B^* := \mathcal{T}_{c_2}(g\#\mu, \nu).$
	
	First, we prove $A^* \leq B^*$. By \cite{santambrogio2015optimal}, Theorem
	1.7, there exists $\pi^* \in \Pi(\mu, \nu)$ such that $B^* = B(\pi^*)$.
	Define $\rho \in \Pi(\mu, \gamma)$ a measure such that for each test
	function $f$, 
	$$\Int{\R^k\times(\R^k\times\R^d)}{}f(x, y_1, y_2)\dd\rho(x, y_1, y_2) =
	\Int{\R^k\times\R^d}{}f(y_1, y_1, y_2)\dd\pi^*(y_1, y_2),$$ or symbolically
	``$\rho(\dd x \dd y_1 \dd y_2) = \delta_{y_1}(\dd x)\pi^*(\dd y_1 \dd
	y_2)$". Then, since $h(c_1(y_1, y_1), c_2(g(y_1), y_2))=c_2(g(y_1), y_2)$,
	we have
	$$A^* \leq A(\rho) = \Int{\R^k\times(\R^k\times\R^d)}{}h(c_1(y_1, y_1),
	c_2(g(y_1), y_2))\dd\pi^*(y_1, y_2) = \Int{\R^k\times\R^d}{}c_2(g(y_1),
	y_2)\dd\pi^*(y_1, y_2)=B^*.$$ Now for $A^* \geq B^*$, we let $\rho \in
	\Pi(\mu, \gamma)$. Using $h(u,v)\geq v,$ we have
	$$A(\rho) = \Int{\R^k\times(\R^k\times\R^d)}{}h(c_1(x, y_1), c_2(g(x),
	y_2))\dd\rho(x, y_1, y_2) \geq \Int{\R^k\times(\R^k\times\R^d)}{}c_2(g(x),
	y_2)\dd\rho(x, y_1, y_2).$$ Again, we can define $\pi \in \Pi(\mu, \nu)$
	such that for any test function $f$, 
	$$\Int{\R^k\times\R^d}{}f(x, y_2)\dd\pi(x, y_2) =
	\Int{\R^k\times(\R^k\times\R^d)}{}f(x, y_2)\dd\rho(x, y_1, y_2),$$ and
	notice 
    $$\Int{\R^k\times(\R^k\times\R^d)}{}c_2(g(x), y_2)\dd\rho(x, y_1, y_2) =
	B(\pi) \geq B^*, $$ which yields $A^* \geq B^*$.
\end{proof}

For example, the cost $C(x, y) = \|x-y\|_2^2$ satisfies these conditions (with
$h(u, v) = u+v$), and thus the problems \cref{eqn:CATM} and \cref{eqn:CATP} are
equivalent. This is still the case for costs of the form $C = \|\cdot -
\cdot\|_p^{qp}$ for $p\geq 1$ and $q>0$, in which case one takes $h(u, v) =
(u^{1/q} + v^{1/q})^q$. For $C((x_1, x_2), (y_1, y_2)) = \|(x_1, x_2) - (y_1,
y_2)\|_{\infty}^p$, this is also the case with $c_{1}(x,y) = c_{2}(x,y) =
\|x-y\|_{\infty}^p$ and $h(u,v)=\max(u,v)$.

In \blue{contrast}, a possible choice of norm on the product space is
$\|x\|_\Sigma = (x^\top \Sigma^{-1}x)^{1/2}$ for $\Sigma$ symmetric
positive-definite. This choice is of interest since the cost $C((x_1, x_2),
(y_1, y_2)) = \|(x_1, x_2) - (y_1, y_2)\|_{\Sigma}^2$ is quadratic (which is
desirable for numerics), but does not satisfy the equivalence condition from
\cref{prop:GSSNB_equivalence} as soon as $\Sigma$ is not block-diagonal.

In \cref{fig:illu_plan_approximation}, we illustrate the plan approximation
problem for the quadratic cost for two different plans: the Entropic Optimal
Transport plan \cite{peyre2019computational} and the Gaussian Mixture Model OT
plan \cite{delon2020wasserstein}. The numerics where done using the tools
presented in \cref{sec:numerics_grad_convex}. Note that the plan approximation
is equivalent to a map problem in this case, and has a particular structure due
to the one-dimensional setting, hence we emphasise that
\cref{fig:illu_plan_approximation} is merely an illustration of the problem at
hand.

\begin{figure}[ht]
    \centering
    \begin{adjustbox}{valign=c}
        \begin{subfigure}[b]{0.45\textwidth}
            \centering
            \includegraphics[width=\textwidth]{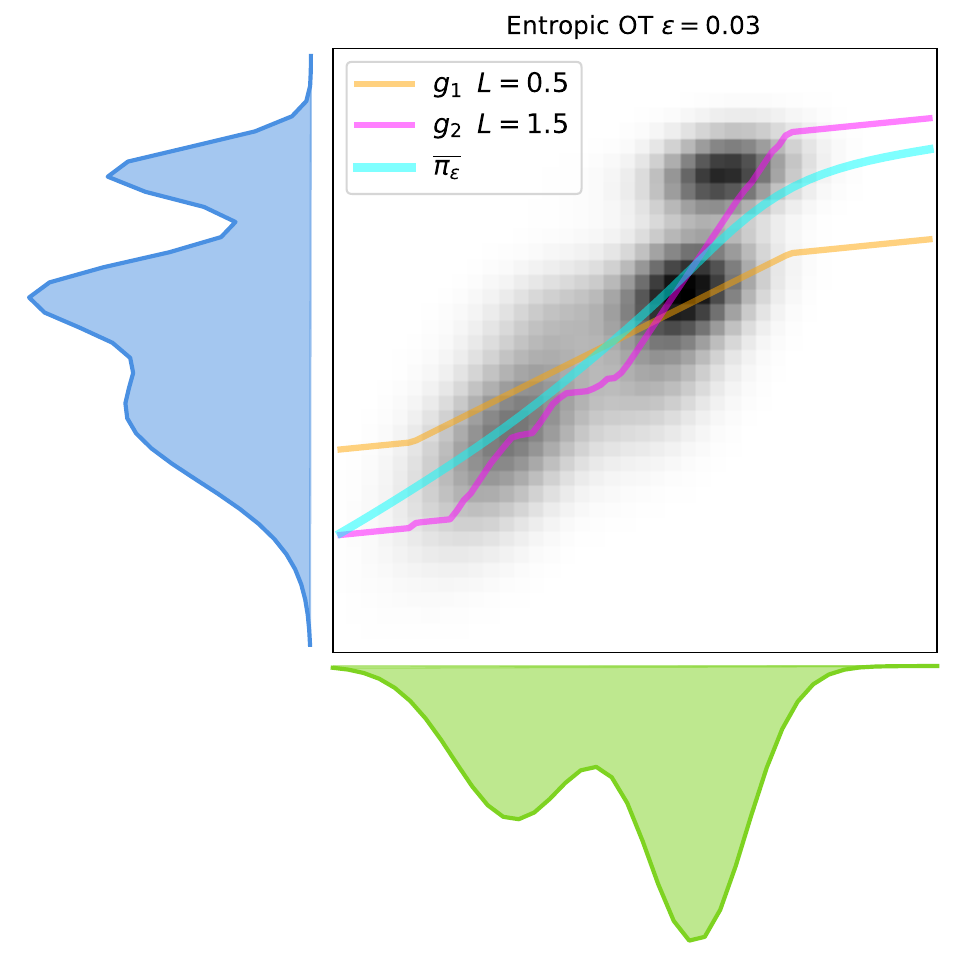}
            \caption{plan approximation solutions for the Entropic-OT plan
            \cite{peyre2019computational}.}
        \end{subfigure}
    \end{adjustbox}
    \hfill
    \begin{adjustbox}{valign=c}
        \begin{subfigure}[b]{0.45\textwidth}
            \centering
            \includegraphics[width=\textwidth]{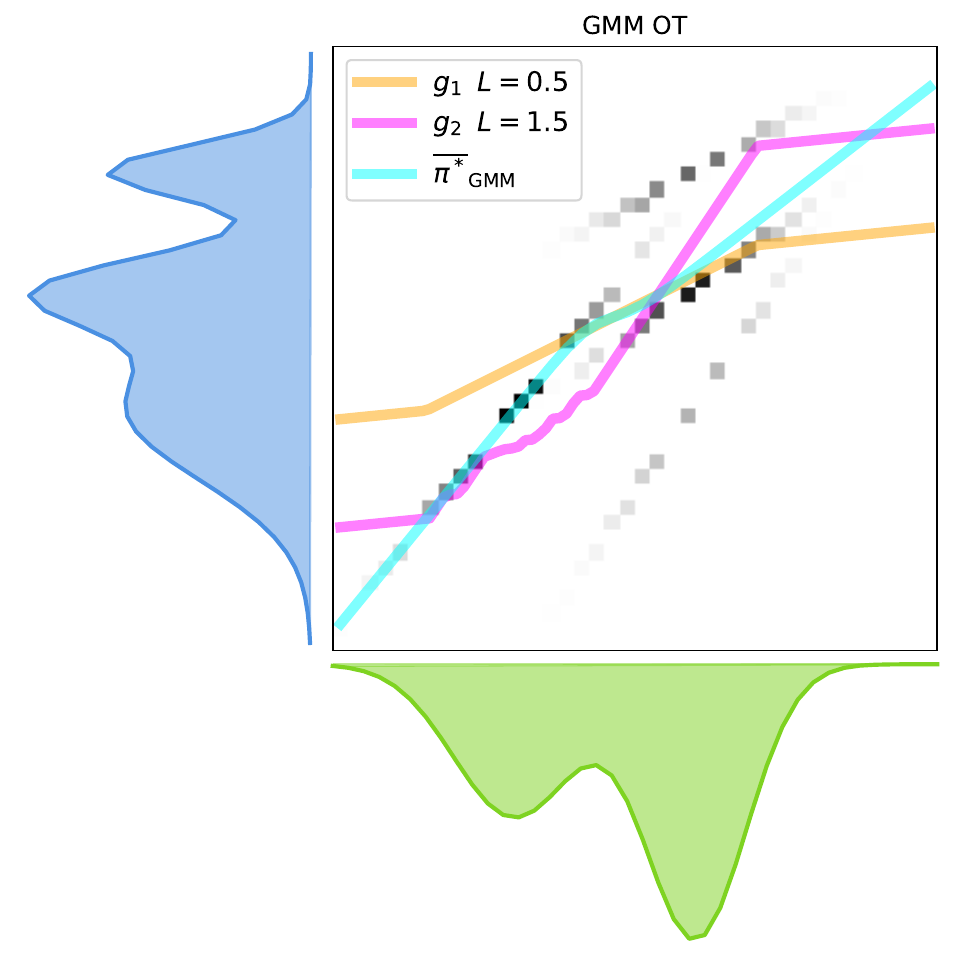}
            \caption{Illustration of plan approximation solutions for the GMM-OT
            plan \cite{delon2020wasserstein}.}
        \end{subfigure}
    \end{adjustbox}
    \caption{Illustration of solutions of plan approximation problems
    (\cref{eqn:CATP}), for two different plans between Gaussian Mixtures. We
    compare the plans with $L=1/2$ and $L=3/2$-Lipschitz solutions, as well as
    to the barycentric projection of the given plans (see
    \cref{sec:proj_of_barycentric}).}
    \label{fig:illu_plan_approximation}
\end{figure}
\section{Alternate Minimisation in the Squared Euclidean
Case}\label{sec:alternate_minimisation}

The map problem \cref{eqn:CATM} is a minimisation problem over $\pi\in \Pi(\mu,
\nu)$ and $g\in G$:
$$\underset{g\in G}{\min}\ \underset{\pi\in\Pi(\mu, \nu)}{\min}\
\int_{\X\times\R^d}c(g(x), y)\dd\pi(x,y). $$ In this section, we study this
alternate minimisation problem in the case where $c(x,y) = \|x-y\|_2^2$
\blue{and $\X=\R^d$, thus with maps $g: \R^d\longrightarrow \R^d$.} 

When $\pi\in \Pi(\mu, \nu)$ is fixed, the sub-problem has the particular structure
\begin{equation}\label{eqn:prob_L2_pi_fixed}
	\underset{g\in G}{\min}\ \int_{\X\times\R^d}\|g(x)-y\|_2^2\dd\pi(x,y).
\end{equation}
To ensure the finiteness of the cost, we assume that \blue{$\mu, \nu \in
\mathcal{P}_2(\R^d)$}. We shall see in \cref{sec:proj_of_barycentric} that the
problem in \cref{eqn:prob_L2_pi_fixed} is equivalent to the $L^2$ projection
of the barycentric map $\oll{\pi}$ onto the set $G$, provided that $G$ is a
convex and closed subset of $L^2(\mu)$.

When $g\in G$ is fixed, the problem reads
\begin{equation}\label{eqn:prob_L2_g_fixed}
	\underset{\pi \in \Pi(\mu, \nu)}{\min}\ \int_{\X\times\R^d}\|g(x)-y\|_2^2\dd\pi(x,y),
\end{equation}
and can be seen from two different viewpoints: either as the squared Euclidean
optimal transport problem between $g\#\mu$ and $\nu$ (i.e. $\W_2^2(g\#\mu,
\nu)$), or as the optimal transport problem with cost $c(x,y) := \|g(x)-y\|_2^2$
between $\mu$ and $\nu$. If $g\#\mu$ is absolutely continuous and $\nu$ is
discrete, then \cref{eqn:prob_L2_g_fixed} is a semi-discrete OT problem. We
provide sufficient conditions on $g$ for this to be the case in
\cref{sec:continuous_to_discrete}.

This alternate minimisation viewpoint poses a natural question: if $\pi := \pi^*
\in \Pi^*(\mu, \nu)$ is an optimal plan between $\mu$ and $\nu$ \blue{for the
quadratic cost $c(y,y') := \|y-y'\|_2^2$}, does the following equality holds?
\begin{equation}\label{eqn:alternate_equivalence_question}
	\underset{g\in G}{\argmin}\ \underset{\pi \in \Pi(\mu, \nu)}{\min}\ \int_{\X\times \R^d} \|g(x)-y\|_2^2\dd\pi(x,y)\;\eqquestion\; \underset{g\in G}{\argmin}\ \int_{\X\times\R^d} \|g(x)-y\|_2^2\dd\pi^*(x,y).
\end{equation}

In \cref{sec:1D}, we prove that this equality holds in the one-dimensional case
$\X=\R^d=\R$ and if $G$ is a subclass of non-decreasing functions, thus
generalizing a result of~\cite{paty2020regularity}. We also provide a
counter-example of this property when $d\geq 2$ in~\cref{sec:CE_2d}.

\subsection{Projection of the Barycentric Map}\label{sec:proj_of_barycentric}

In this section, we will show that the sub-problem with $\pi\in \Pi(\mu, \nu)$
fixed can be written as the following $L^2$ projection problem:
$$\underset{g\in G}{\argmin}\ \int_{\X\times\R^d}\|g(x)-y\|_2^2\dd\pi(x,y) =
\underset{g\in G}{\argmin}\ \|g-\oll{\pi}\|_{L^2(\mu)}^2, $$ where $\oll{\pi}$
is the barycentric projection of $\pi$ (defined below), and $L^2(\mu)$ is a
shorthand for $L^2(\mu; \R^d)$, the space of measurable functions $T:
\X\longrightarrow \R^d$ such that $\int_\X \|T(x)\|_2^2\dd\mu(x)<+\infty$. We
begin by briefly introducing the notion of barycentric projection.

The \textit{barycentric projection} of $\pi$ is the map $\oll{\pi}: \R^d
\longrightarrow \R^d$ defined for $\mu$-almost all $x\in\R^d$ by
\begin{equation}\label{eqn:bar_proj}
	\oll{\pi}(x) = \mathbb{E}_{(X, Y)\sim \pi}[Y|X=x].
\end{equation}
As illustrated in \cref{fig:barycentric_projection}, if $\pi$ admits a
disintegration with respect to its first marginal $\mu$ of the form $\pi(\dd
x\dd y) = \pi_{x}(\dd y) \mu(\dd x)$, then
$$\oll{\pi}(x) = \Int{\R^d}{}y\, \dd\pi_x(y). $$

\begin{figure}[ht]
    \centering
    \includegraphics[width=.4\linewidth]{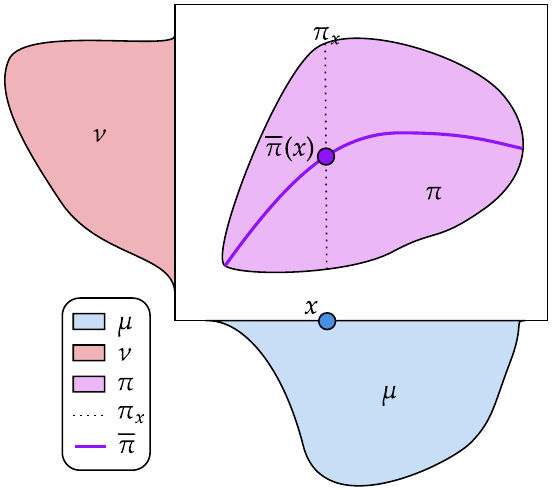}
    \caption{Illustration of a barycentric projection. The disintegration of the coupling $\pi$ with respect to its first marginal $\mu$ at $x$ is the measure $\pi_x$ concentrated on the dotted line. The barycentric projection of $\pi$ evaluated at $x$ is the mean of the measure $\pi_x$.}
    \label{fig:barycentric_projection}
  \end{figure}
  
Since the conditional expectation minimises the $L^2$ distance, we also have
\begin{equation}\label{eqn:bar_proj_minL2}
	\oll{\pi} = \underset{T \in L^2(\mu)}{\argmin}\Int{\R^{2d}}{}\|y-T(x)\|_2^2\dd\pi(x,y),
\end{equation}
where the equality is to be understood in $L^2(\mu)$. Another interesting property is that if $\pi=\pi^* \in \Pi^*(\mu, \nu)$ is an
optimal transport plan between $\mu$ and $\nu$ with respect to the squared
Euclidean distance cost, then by \cite{ambrosio2005gradient}, Section 6.2.3,
there exists $\varphi: \R^d \longrightarrow \R$ convex such that for
$\pi^*$-almost-every $(x,y) \in \R^{2d}$, we have $y \in \partial \varphi(x)$,
where $\partial \varphi(x)$ denotes the \textit{Fréchet sub-differential} of
$\varphi$:
$$y \in \partial\varphi(x) \Longleftrightarrow \underset{z\longrightarrow
x}{\liminf}\ \cfrac{\varphi(z) - \varphi(x) - \langle y, z-x\rangle}{\|x-z\|_2}
\geq 0. $$ Since the Fréchet sub-differential of a convex function is convex, it
follows that for $\mu$-almost every $x \in \R^d,\; \oll{\pi^*}(x) \in \partial
\varphi(x)$.

If we require constraints on $T$ in \cref{eqn:bar_proj_minL2}, we obtain exactly
the sub-problem of the map problem with a fixed plan $\pi$
(\cref{eqn:prob_L2_pi_fixed}), which we reproduce below:
$$\underset{g \in G}{\argmin}\ \Int{\X\times \R^d}{}\|g(x)-y\|_2^2\dd\pi(x,
y).$$ For this reason, we call this problem the Constrained Barycentric Map
problem. A consequence of the proof of \cref{thm:existence_G_c} is that this
problem has a solution. If $G$ is a convex set and closed in $L^2(\mu)$, then
existence and uniqueness are guaranteed by the Hilbert projection Theorem. Since
$\oll{\pi}$ minimises the $L^2$ distance, it is a solution of
\cref{eqn:prob_L2_pi_fixed} if it is in $G$.

Using the fact that the barycentric projection is an $L^2$ projection
(\cref{eqn:bar_proj_minL2}), one may re-write the Projected Barycentric Map
Problem \cref{eqn:prob_L2_pi_fixed} as an $L^2$ minimisation with respect to the
barycentric projection. In
\cref{prop:constrained_bar_L2_proj_on_barycentric_projection}, we need not
assume that $\X=\R^d$, but we shall apply it later in \cref{sec:1D} to the case
$\X=\R$.

\begin{prop}\label{prop:constrained_bar_L2_proj_on_barycentric_projection} Let
	$\pi\in\Pi(\mu, \nu)$ and $f: \X \longrightarrow \R^d$ be a measurable
	function. Then one has
	\begin{equation}\label{eqn:snb_cost_reformulation}
		\Int{\X\times\R^d}{}\|f(x)-y\|_2^2\dd\pi(x,y) = \Int{\X}{}\|f(x)-\oll{\pi}(x)\|_2^2\dd\mu(x) + \Int{\X\times \R^d}{}\|y-\oll{\pi}(x)\|_2^2\dd\pi(x,y),
	\end{equation}
	and as a result, the Projected Barycentric Map problem
	\cref{eqn:prob_L2_pi_fixed} is equivalent to the problem
	\begin{equation}\label{eqn:snb_L2_reformulation}
		\underset{g\in G}{\argmin}\ \Int{\X}{}\|g(x)-\oll{\pi}(x)\|_2^2\dd\mu(x).
    \end{equation}

Moreover, the second term on the right-hand side of
\cref{eqn:snb_cost_reformulation}  only depends on
$\oll{\pi}$ and the measures $\mu, \nu$ (it doesn't depend on
$\pi$). More precisely, we have 
$$ \Int{\X\times \R^d}{}\|y-\oll{\pi}(x)\|_2^2\dd\pi(x,y) = m_2(\nu) -
m_2(\oll{\pi}\#\mu), $$
where $m_2(\rho) := \int\|x\|_2^2\dd\rho(x)$ for a positive measure $\rho$.
\end{prop}

\begin{proof}
	Denote  $J$ the left-hand-side of \cref{eqn:snb_cost_reformulation}, and
	compute (taking the expectation under $(X,Y)\sim\pi$)
	\begin{align*}
		J &= \E{\|Y-f(X)\|_2^2} 
		= \E{\|Y - \oll{\pi}(X) + \oll{\pi}(X) + f(X)\|_2^2} \\
		&= \E{\|Y - \oll{\pi}(X)\|_2^2} + \E{\|\oll{\pi}(X) + f(X)\|_2^2} + 2\E{(Y - \oll{\pi}(X))^T(\oll{\pi}(X) + f(X))},
	\end{align*}
	then since $\oll{\pi}(X)$ is the orthogonal projection of $Y$ onto the set
	of random variables that are functions of $X$, the inner product $\E{(Y -
	\oll{\pi}(X))^T(\oll{\pi}(X) + f(X))}$ is zero, yielding
	\cref{eqn:snb_cost_reformulation}.

        We can expand the norm in the second term of the right-hand side of
	\cref{eqn:snb_cost_reformulation} using $m_2(\nu)$ the second moment of
	$\nu$ and get
	\begin{align*}
		\Int{\X\times \R^d}{}\|y-\oll{\pi}(x)\|_2^2\dd\pi(x,y)
          &= m_2(\nu) -  2\Int{\X\times \R^d}{}y\cdot \oll{\pi}(x)\dd\pi(x,y) + \Int{\X}{}\|\oll{\pi}(x)\|_2^2\dd\mu(x).
	\end{align*}
	Writing the disintegration of $\pi$ w.r.t. $\mu$ as $\pi(\dd x, \dd y) =
	\pi_x(\dd y)\mu(\dd x)$, we re-write the second term as
	$$\int_{\X\times \R^d}y\cdot \oll{\pi}(x)\dd\pi(x,y) =
	\int_{\X}\oll{\pi}(x)\cdot \left(\int_{\R^d}y\dd\pi_x(y)\right) \dd\mu(x)
	= \Int{\X}{}\oll{\pi}(x)\cdot \oll{\pi}(x)\dd\mu(x) =
	m_2(\oll{\pi}\#\mu).$$ Putting our computations together yields
	$$\Int{\X\times \R^d}{}\|y-\oll{\pi}(x)\|_2^2\dd\pi(x,y) = m_2(\nu) -
	m_2(\oll{\pi}\#\mu).$$
\end{proof}
\begin{remark}[Ties to the Convex Least Squares
Estimator~\cite{manole2021plugin}.]
In \cite{manole2021plugin}, Manole et al. study the statistical properties of
various estimators of Optimal Transport maps, assuming some regularity on the
input distributions. Specifically, they introduce the so-called Convex Least
Squares Estimator: given $\hat\mu_n := \frac{1}{n}\sum_{i=1}^n \delta_{x_i}$
with the $(x_i)$ being i.i.d. samples of $\mu$ and $\hat\nu_m :=
\frac{1}{m}\sum_{j=1}^m \delta_{y_j}$ with the $(y_j)$ i.i.d. samples of $\nu$,
the estimator is defined as
\begin{equation}\label{eqn:cvx_least_squares_estimator}
	\hat T_{n,m} = \nabla \hat\pot,\quad \hat\pot \in \underset{\pot \in \Phi_\gradlip}{\argmin}\ \Sum{i=1}{n}\Sum{j=1}{m}\hat{\pi}^*_{i,j}\|\nabla\varphi(x_i) - y_j\|_2^2,
\end{equation}
where $\hat\pi^*$ is an optimal transport plan between $\hat\mu_n$ and
$\hat\nu_m$, and where $\Phi_\gradlip$ is the set of $\mathcal{C}^1$ convex
functions from $\Omega\subset \R^d$ to $\R$ with a $\gradlip$-Lipschitz
gradient.  
Notice that~\cref{eqn:cvx_least_squares_estimator} is a Constrained Barycentric
Projection problem \cref{eqn:prob_L2_pi_fixed} with a specific (discrete)
transport plan $\hat\pi^*$, chosen to be the optimal transport plan between
$\hat\mu_n$ and $\hat\nu_m$, and with the particular class $G := \Funs$
(introduced in \cref{sec:class_grad_convex}).  
\end{remark}

\subsection{Equivalence to a Constrained Barycentric Projection in Dimension
1}\label{sec:1D}

In this section, we shall prove that the Constrained Approximate Transport Map
problem (\cref{eqn:CATM}) is equivalent to the Constrained Barycentric
Projection Problem (\cref{eqn:prob_L2_pi_fixed}) for the quadratic cost in
dimension $1$. This provides a positive answer to the question raised in
\cref{eqn:alternate_equivalence_question} in this particular case. The idea
behind this equivalence stems from the fact that in dimension one, optimal
transport maps are non-decreasing, and the composition of two optimal transport
maps remains an optimal transport map. 

\begin{prop}\label{prop:1d_equiv_proj} For $\mu, \nu \in \mathcal{P}_2(\R)$, and
	$G$ a subclass of the non-decreasing functions $g: \R\longrightarrow \R$
	such that $g\#\mu\in\mathcal{P}_2(\R)$, we have the equality
	\begin{equation}\label{eqn:1d_equiv_proj}
		\underset{g\in G}{\argmin}\ \W_2^2(g\#\mu, \nu) = \underset{g\in
		G}{\argmin}\ \|g - \oll{\pi^*}\|_{L^2(\mu)}^2,
	\end{equation}
	where $\pi^*$ is an optimal transport plan between $\mu$ and $\nu$ for the
	squared Euclidean cost.
\end{prop}

\cref{prop:1d_equiv_proj} generalises \cite{paty2020regularity}
Proposition 1, which proves the same equivalence for a specific class of
functions $G$, and assuming $\mu$ to be either discrete or absolutely continuous
with respect to the Lebesgue measure.

The proof of \cref{prop:1d_equiv_proj} hinges on
\cref{lemma:push_forward_increasing_quantile}, which is intuitive in the
absolutely continuous or discrete case, but a bit more technical in full
generality. We write below the cumulative distribution function of a probability
measure $\rho$ as $F_\rho:= x \longmapsto \rho((-\infty, x])$. Since it is
non-decreasing, we can define its \textbf{right-inverse} as (using the notation
$\oll{\R}:= \R\cup\{-\infty, +\infty\}$):
$$\rinv{F_\rho}: \R \longrightarrow \oll{\R}: \quad\forall p \in \R,\;
\rinv{F_\rho}(p) := \inf\left.\left\{x \in \R\  \right|\ F_\rho (x)\geq
p\right\}.$$
\begin{lemma}\label{lemma:push_forward_increasing_quantile} Let $\mu \in
	\mathcal{P}(\R)$ and $g:\R\longrightarrow\R$ be a non-decreasing function,
	we have the following almost-everywhere change of variables formula for the
	quantile functions of $g\#\mu$ and $\mu$:
	$$\rinv{F_{g\#\mu}} = g \circ \rinv{F_\mu},\quad \Leb_{[0,
	1]}\text{-almost-everywhere}.$$
\end{lemma}

\begin{proof}
	The proof is provided in \cref{appendix:quantile_inverses}. 
\end{proof}

\begin{proof}[Proof of \cref{prop:1d_equiv_proj} ]
	Let $g\in G$. By \cite{santambrogio2015optimal} Proposition 2.17 and by
	\cref{lemma:push_forward_increasing_quantile} successively, we have
	\begin{align*}
		\W_2^2(g\#\mu, \nu) &=\int_0^1|\rinv{F_{g\#\mu}}(p)-\rinv{F_\nu}(p)|^2\dd p = \int_0^1|g\circ \rinv{F_\mu}(p) - \rinv{F_\nu}(p)|^2\dd p 
		= \int_{\R^2}|g(x)-y|^2\dd\pi(x,y),
	\end{align*}
	where $\pi := (\rinv{F_\mu}, \rinv{F_\nu})\#\Leb_{[0,1]}$, which by
	\cite{santambrogio2015optimal} Theorem 2.9 is the unique optimal plan
	between $\mu$ and $\nu$ for the squared Euclidean cost. We apply
	\cref{prop:constrained_bar_L2_proj_on_barycentric_projection}, which yields
	$$\W_2^2(g\#\mu, \nu) = \int_{\R^2}|g(x)-y|^2\dd\pi(x,y) =
	\int_{\R}|g(x)-\oll{\pi}(x)|^2\dd\mu(x) + m_2(\nu) - m_2(\oll{\pi}\#\mu).$$
	Given the expression of the right-hand-side above, we conclude that
	$$\underset{g\in G}{\argmin}\ \W_2^2(g\#\mu, \nu) = \underset{g\in
	G}{\argmin}\ \|g - \oll{\pi^*}\|_{L^2(\mu)}^2,$$ (for any optimal transport
	plan $\pi^*$ between $\mu$ and $\nu$ for the squared Euclidean cost, and we
	have even remarked that such a plan is in fact unique) since the costs are
	equal up to a constant independent of $g$.
\end{proof}

\subsection{Counter-Example to Equivalence to Constrained Barycentric Projection
in Dimension 2}\label{sec:CE_2d}

In this section, we provide a negative example to the question formulated in
\cref{eqn:alternate_equivalence_question}, namely that
$$\underset{g\in G}{\argmin}\ \W_2^2(g\#\mu, \nu) \neq \underset{g\in
G}{\argmin}\ \|g - \oll{\pi^*}\|_{L^2(\mu)},$$ where $\pi^*$ is an optimal
transport plan (for the squared Euclidean cost) between $\mu$ and $\nu$, in
dimension $d\geq 2$. We take $G$ to be the class of \textit{monotone} continuous
functions $g: \R^2 \longrightarrow \R^2$, which is to say that
$$\forall x,y\in \R^2,\; \langle g(x) - g(y), x - y\rangle \geq 0. $$ Note that
gradients of convex functions are monotone, but the converse does not hold.
For $(a,b,x)\in (0,+\infty)^3$, we consider the following measures:
$$\mu := \frac{2}{3}\delta_{(0, 0)} + \frac{1}{3}\delta_{(x, 0)} \text{ and }
 \nu := \frac{2}{3}\delta_{(0, 0)} + \frac{1}{3}\delta_{(-a, b)}.$$ 
There is a unique optimal transport plan  $\pi^*$ between $\mu$ and $\nu$, given
by
$$\pi^* = \frac{1}{3}\delta_{(0, 0)\otimes (0, 0)} + \frac{1}{3}\delta_{(0,
0)\otimes(-a,b)} + \frac{1}{3}\delta_{(x, 0)\otimes(0, 0)}. $$ Its barycentric
projection is characterised by the following equation
$$\oll{\pi^*}(0, 0)=(-a/2, b/2) \text{ and } \oll{\pi^*}(x, 0) = (0, 0).$$ We
now consider the problem $\underset{g\in G}{\min}\|g-\oll{\pi^*}\|_{L^2(\mu)}$.
A solution of this problem is characterised by its values on the support of
$\mu$, and one may reduce the problem to an optimisation over $g(0, 0)$ and
$g(x, 0)$, with the monotonicity constraint $\langle g(0, 0) - g(x, 0), (0, 0) -
(x, 0) \rangle \geq 0$. Since $\oll{\pi^*}$ itself verifies this condition, it
is the only solution (in the sense of $L^2(\mu)$). We conclude
$$\underset{g\in G}{\argmin}\ \|g - \oll{\pi^*}\|_{L^2(\mu)}^2 =
\left\lbrace\oll{\pi^*} \right\rbrace. $$ We now show that the problem
$\underset{g \in G}{\argmin}\ \W_2^2(g\#\mu, \nu)$ has a different solution set.
First, we compute 
$$\W_2^2(\oll{\pi^*}\#\mu, \nu)=\frac{a^2+b^2}{6}.$$ However, if we introduce
$g\in G$ such that $g(0, 0)=(0, 0)$ and $g(x, 0)=(0, b)$, we have
$$\W_2^2(g\#\mu, \nu) = \frac{a^2}{3}. $$ For instance, $(a,b,x) := (1, 10, 1)$
yields 
$$\W_2^2(\oll{\pi^*}\#\mu, \nu)=\frac{a^2+b^2}{6} = \frac{101}{6} >
\W_2^2(g\#\mu, \nu)=\frac{a^2}{3} = \frac{1}{3}.$$ We illustrate the point
configurations for $(a, b, x) := (1, 3, 1)$ in \cref{fig:CE_2d}.
\begin{figure}[ht]
	\centering
	\includegraphics[width=.6\linewidth]{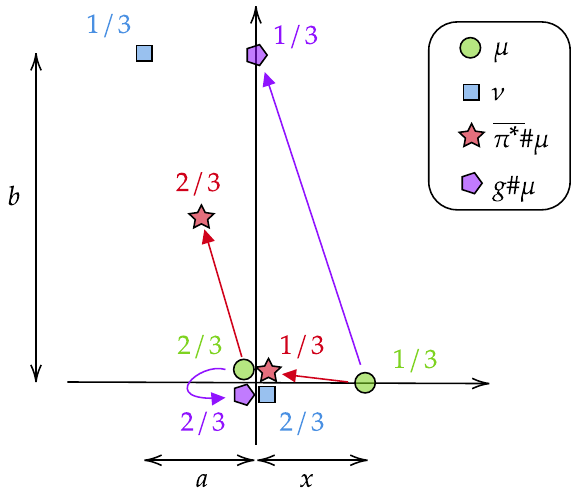}
	\caption{Illustration of the two-dimensional counter-example
          to the equivalence of the map problem to the $L^2$
          projection of the barycentric projection. The four points
          close to $(0, 0)$ are represented with an offset for
          legibility, and represent four points equal to $(0, 0)$
          exactly.} 
	\label{fig:CE_2d}
\end{figure}

\section{Discrete Measures and Numerical Methods}\label{sec:discrete}

In this section, we consider some numerical methods to solve the approximate map
problem for some specific function classes. \blue{To prepare for convergence
results, we dedicate \cref{sec:semi_algebraic} to regularity properties of the
transport cost in the discrete case.} In \cref{sec:numerics_grad_convex}, we
present methods in the case where $G$ is the class of $\gradlip$-Lipschitz
gradients of $\strcvx$-strongly convex potentials (presented in
\cref{sec:class_grad_convex}). For the squared Euclidean cost, these methods
were introduced in \cite{paty2020regularity}, using convex interpolation results
from \cite{taylor2017convex}. \blue{In \cref{sec:kernels}, we consider a simple
kernel method which solves a regularised version of \cref{eqn:CATM}. This type
of method hinges on the fact that kernel methods yield a finite-dimensional
parametrisation of the function $g$, and allows for provably convergent gradient
descent methods. In \cref{sec:sgd_nn_map}, we consider a Stochastic Gradient
Descent method for the case where the map $g$ is a Neural Network. Finally, in
\cref{sec:colour_transfer}, we illustrate the use of the methods presented in
this section on the problem of colour transfer.}

\blue{\subsection{Regularity of Discrete Optimal Transport
Costs}\label{sec:semi_algebraic}

To study the convergence of sub-gradient descent methods theoretically, we will
introduce standard notions from non-smooth non-convex analysis, in particular a
specific generalisation of sub-gradients, which are in practice computed by
automatic differentiation. A central notion in this analysis will be the notion
of semi-algebraicity, which we remind in \cref{def:semi_algebraic} (and refer to
\cite{Wakabayashi_semialgebraic} and \cite{bolte2021conservative} for more
details).

\begin{definition}\label{def:semi_algebraic} A set $S \subset \R^d$ is said to
    be semi-algebraic if it can be written under the form $S =
    \cup_{n=1}^N\cap_{m=1}^M S_{n,m}$, where each $S_{n,m}$ is either of the
    form $\{x\in \R^d: p_{n,m}(x) = 0\}$ or $\{x\in \R^d: p_{n,m}(x) \geq 0\}$,
    where $p_{n,m}$ is a $d$-variate polynomial with real coefficients.

    A function $f: \R^{d_1} \longrightarrow \R^{d_2}$ is semi-algebraic if its
    graph $\{(x, f(x)) : x \in \R^{d_1}\}$ is a semi-algebraic set.

    A multifunction $f: \R^{d_1} \rightrightarrows \R^{d_2}$ is semi-algebraic
    if its graph $\{\{x\}\times f(x) : x \in \R^{d_1}\}$ is a semi-algebraic
    set.
\end{definition}

Another central notion will be a generalisation of the notion of gradient for
locally Lipschitz functions called the Clarke differential.

\begin{definition}\label{def:Clarke} Given a locally Lipschitz function $f:
    \R^{d} \longrightarrow \R$, its Clarke sub-gradient at $x \in \R^{d}$ is the
    set
    $$\partial_C f(x) = \conv\left\{\underset{t \longrightarrow
    +\infty}{\lim}\nabla f(x_t) : x_t \xrightarrow[t\longrightarrow +\infty]{}
    x,\; x_t \in D_f\right\},$$ where $D_f$ is the set of differentiability of
    $f$ and $\conv$ denotes the convex envelope. 
\end{definition}

In \cref{prop:discrete_kanto_clarke_semi_algebraic}, we show that the discrete
OT cost is semi-algebraic and locally Lipschitz as a function of the cost
matrix, and that its Clarke sub-gradient is itself semi-algebraic.

\begin{prop}\label{prop:discrete_kanto_clarke_semi_algebraic} Consider weights
    $a \in \Delta_n$ (the $n$-simplex) and $b\in \Delta_m$ (the $m$-simplex),
    and the discrete Kantorovich cost function
    $$W(a, b, \cdot) := \app{\R^{n\times m}}{\R}{M}{\underset{\pi \in
    \Pi(a,b)}{\min}\ \pi \cdot M}. $$ Then the map $W(a, b, \cdot)$ is
    semi-algebraic, Lipschitz, and its Clarke sub-gradient is semi-algebraic and
    writes for $M \in \R^{n\times m}$:
    $$\partial_C W(a, b, \cdot)(M) = \underset{\pi \in \Pi(a,b)}{\argmin}\
    \pi\cdot M. $$
\end{prop}

\begin{proof}
    Writing the extremal points of the polytope $\Pi(a,b)$ as $(\pi_i)_{i=1}^N$,
    the function $W(a, b, \cdot)$ can be written as a finite minimisation over
    $\pi \in (\pi_i)_{i=1}^N$ of linear functions of $M$, hence $W(a, b, \cdot)$
    is semi-algebraic. The fact that $W(a, b, \cdot)$ is Lipschitz is also its
    consequence of its expression as a minimum of a finite amount of linear
    maps. The expression of sub-gradients of $W(a, b, \cdot)$ is a consequence
    of Danskin's Theorem \cite{danskin1966theory}. Finally, the semi-algebraic
    property of $\partial_C W(a, b, \cdot)$ is a consequence of the fact that
    $W(a, b, \cdot)$ is semi-algebraic and locally Lipschitz, or can
    alternatively be seen using the extremal point method as done for $W(a, b,
    \cdot)$.
\end{proof}

In \cref{lemma:semi_algebraic_props} we remind some useful properties of
semi-algebraic maps that we will use later on. In particular, semi-algebraic
maps are \textit{generalised differentiable} (see \cite{ermoliev1997stochastic}
Definition 3.1), which can be understood as a generalised first-order Taylor
expansion.

\begin{lemma}\label{lemma:semi_algebraic_props} Any locally Lipschitz
    semi-algebraic map $f: \R^d \longrightarrow \R$ is generalised
    differentiable (see \cite{ermoliev1997stochastic} Definition 3.1) and its
    set of critical values $f\{x \in \R^{d} : 0 \in\partial_C f(x)\}$ is finite.
\end{lemma}
\begin{proof}    
    Since $f$ is semi-algebraic and locally Lipschitz, by \cite{bolte2009tame}
    Theorem 3.6 it is semi-smooth, which in turn implies generalised
    differentiability (by \cite{mikhalevich2024methods} Theorem 1.4). Next, by
    definable Morse-Sard (from \cite{bolte2021conservative} Theorem 5), the set
    of critical values $f\{x \in \R^{d} : 0 \in\partial_C f(x)\}$ is finite. 
\end{proof}
} \subsection{Numerical Method for Gradients of Convex
Functions}\label{sec:numerics_grad_convex}

In this section, we present numerical methods to solve the approximate problem
in the case of the function class $\Funs$ of functions $g: \R^d \longrightarrow
\R^d$ that is $\gradlip$-Lipschitz and gradient of an $\strcvx$-strongly convex
function $\pot \in \mathcal{C}^1(\R^d, \R)$, on each part $E_k$ of the fixed
partition $\mathcal{E}$.  We already introduced this class in
\cref{sec:class_grad_convex}, and it was first considered in the context of map
problems by \cite{paty2020regularity}. The numerical methods will aim to solve
the problem
\begin{equation}\label{eqn:SSNB_Tc}
	\underset{\pot \in \Funs}{\argmin}\ \Tc(g\#\mu, \nu),
\end{equation} 
with a particular emphasis on the case where $c$ is quadratic, i.e. $c(x, y) =
(x-y)^TQ(x-y) + b^T(x-y)$, where $Q\in S_d^+(\R)$ is a positive-semi-definite
matrix, and $b\in \R^d$. For our numerical questions, we consider the discrete
case
$$\mu = \sum_{i=1}^na_i\delta_{x_i},\quad \nu = \sum_{j=1}^mb_j\delta_{y_j}. $$

Obviously, we need to assume that the measure $\mu$ is compatible with the
partition, which is to say the the $x_i$ are never at the boundary of a part
$E_k$: $\forall i \in \llbracket 1, n \rrbracket,\; x_i \in \left(\cup_k
\partial E_k\right)^c$. The objective in \cref{eqn:SSNB_Tc} only depends on the
values $\pot_i := \pot(x_i)$ and $g_i := g(x_i)$, the immediate question is that
given a candidate $(\pot_i, g_i) \in (\R\times\R^d)^n$, does there exist a
function $g\in \Funs$ of the form $\nabla\pot$ which interpolates these values,
i.e. $g(x_i)=g_i$ and $\pot(x_i)=\pot_i$? This question, which is called
$\Funs$-interpolation, was studied by Taylor
\cite{taylor2017convex}\footnote{Note that (\cite{taylor2017convex}, Theorem
3.14) writes an erroneous $\argmin$ for $\pot_u$: in the light of
(\cite{taylor2017convex}, Remark 3.13), it should instead read $\argmax$,
especially given the fact that the minimisation problem is unbounded.}. We write
$\FunsRd := \Funs$ in the case $\mathcal{E} = \lbrace \R^d \rbrace$, and present
the results in the restricted case where the space is $\R^d$, as opposed to any
vector space.

\begin{prop}\label{prop:taylor_discretisation} (Multiple results from Taylor
\cite{taylor2017convex,taylor2017smooth}). Let $S=(x_i, g_i, \pot_i)_{i \in
\llbracket 1, n \rrbracket} \in (\R^d \times \R^d \times \R)^n$. \blue{The set
$S$ is said to be $\FunsRd$-interpolable (\cite{taylor2017convex}, Definition
3.1) if there exists $\pot \in \FunsRd$ such that $\forall i \in \llbracket 1, n
\rrbracket,\; \nabla\pot(x_i) = g_i$ and $\pot(x_i) = \pot_i$.} Consider the
quadratic function
    \begin{equation}\label{eqn:discrete_qcqp_Q}
        Q(x, x', \pot, \pot', g, g') := \pot-\pot'-\langle g', x-x'\rangle - c_1\|g - g'\|_2^2 - c_2\|x-x'\|_2^2 + c_3\langle g'-g, x' -x \rangle,
    \end{equation}
    for $x, x'\in \R^d,\; \pot,\pot'\in \R,\; g, g'\in \R^d$, with
    $$c_1 := \cfrac{1}{2\gradlip(1-\strcvx/\gradlip)},\quad c_2 :=
    \cfrac{\strcvx}{2(1-\strcvx/\gradlip)},\quad c_3 :=
    \cfrac{\strcvx}{\gradlip(1-\strcvx/\gradlip)}. $$
	\begin{itemize}
		\item (\cite{taylor2017convex}, Theorem 3.8) The set $S$ is
		$\FunsRd$-interpolable if and only if for all $i,j \in \llbracket 1, n
		\rrbracket,$
			\begin{align}\label{eqn:cns-cvx-interpolable}
                Q(x_i, x_j, \pot_i, \pot_j, g_i, g_j) \geq 0.
			\end{align}
		\item (\cite{taylor2017convex}, Theorem 3.14)  For $x \in \R^d$, let:
			\begin{align}\label{eqn:pot_l}
				\pot_{l}(x) = &\underset{t\in \R,\; g\in \R^d}{\min}\ t, \\
				&\mathrm{s.t.}\ \forall j \in \llbracket 1, n \rrbracket,\;
                Q(x, x_j, t, \pot_j, g, g_j) \geq 0;\nonumber
			\end{align}
			\begin{align}\label{eqn:pot_u}
				\pot_{u}(x) = &\underset{t\in \R,\; g\in \R^d}{\max}\ t, \\
				&\mathrm{s.t.}\ \forall i \in \llbracket 1, n \rrbracket,\;
                Q(x_i, x, \pot_i, t, g_i, g) \geq 0.\nonumber
			\end{align}
			If $S$ is $\FunsRd$-interpolable, then any interpolating function
			$\pot$ satisfies $\pot_{l} \leq \pot \leq \pot_{u}$, and the
			potentials $\pot_{l}, \pot_{u}$ are valid interpolations.
	\end{itemize}
\end{prop}

\cref{prop:taylor_discretisation} shows that the constraint on $(\pot_i, g_i)_i$
can be written as a set of quadratic constraints. It follows immediately that
any problem that only depends on the values $g(x_i)$ for a variable $G\in \Funs$
can be written as a problem over $(\pot_i, g_i)_i$ under quadratic constraints,
as stated in \cref{cor:apply_taylor}.

\begin{corollary}\label{cor:apply_taylor} Consider an objective $J: \Funs
    \longrightarrow \R_+$ such that for $g\in \Funs$, the value $J(g)$ can be
    written $J\left(g(x_1), \cdots, g(x_n)\right)$. Then the problem
    \begin{equation}\label{eqn:apply_taylor_continuous_form}
        \underset{g\in \Funs}{\min}\ J(g)
    \end{equation}
    is equivalent to the problem
    \begin{align}\label{eqn:apply_taylor_discrete_form}
        &\underset{\substack{\pot_1, \cdots, \pot_n \in \R\\ 
        g_1, \cdots, g_n \in \R^d}}{\min} J(g(x_1), \cdots, g(x_n)) \\
        &\text{s.t.}\ \forall k \in \llbracket 1, K \rrbracket,\; 
        \forall i,j \in I_k : Q(x_i, x_j, \pot_i, \pot_j, g_i, g_j) \geq 0,
        \nonumber        
    \end{align}
    where $I_k := \lbrace i\in  \llbracket 1, n \rrbracket\ |\ x_i \in
	E_k\rbrace$, and $Q$ is defined in \cref{eqn:discrete_qcqp_Q}. Given a
	solution $(\pot_i^*, g_i^*)_i$ of \cref{eqn:apply_taylor_discrete_form}, any
	solution $\nabla \pot^*$ of \cref{eqn:apply_taylor_continuous_form}
	satisfies $\pot_l \leq \pot^* \leq \pot_u$ on $\cup_k \interior{E_k}$, where
	for $x \in \interior{E_k}$, the bounding potentials and their gradients are
	solutions of:
	\begin{align}\label{eqn:pot_lk}
		(\pot_{l}(x), \nabla \pot_l(x)) = 
        &\underset{t\in \R,\; g\in \R^d}{\argmin}\ t, \\
		&\mathrm{s.t.}\ \forall j \in I_k,\; 
        Q(x, x_j, t, \pot_j^*, g, g_j^*)\geq 0;\nonumber
	\end{align}
	\begin{align}\label{eqn:pot_uk}
		(\pot_{u}(x), \nabla\pot_u(x)) = &\underset{t\in \R,\; g\in \R^d}{\argmax}\ t, \\
		&\mathrm{s.t.}\ \forall i \in I_k,\; Q(x_i, x, \pot_i^*, t, g_i^*, g) \geq 0.\nonumber
	\end{align}
	The potentials $(\varphi_l, \varphi_u)$ themselves are both solutions of
	\cref{eqn:apply_taylor_continuous_form}.
\end{corollary}

Note that the values of the potentials can be chosen arbitrarily on the
boundaries $\partial E_k$.

We can now provide an algorithm for $\argmin_{g\in \Funs}\Tc(g\#\mu, \nu)$
(\cref{eqn:SSNB_Tc}) using \cref{cor:apply_taylor}: the objective is
\begin{equation}\label{eqn:objective_SSNB_Tc_discrete}
    J(g(x_1), \cdots g(x_n)) = \underset{\pi\in \Pi(a, b)}{\min}\ \sum_{i,j}\pi_{i,j}c(g(x_i), y_j),
\end{equation}
and the resulting problem defined in \cref{eqn:apply_taylor_discrete_form} can
be solved by alternating over $\pi$ (solving a discrete Kantorovich problem,
using \texttt{ot.emd} from the PythonOT library, for instance
\cite{flamary2021pot}), and over $(\pot_i, g_i)$, for which the constraints are
quadratic and the objective depends on the cost $c$. For smooth cost, one may
use projected gradient descent, and for (convex) quadratic costs, the problem
becomes a (convex) Quadratically Constrained Quadratic Program (QCQP). In the
case $c(x, y) = \|x-y\|_2^2$, this method is already known, and is the core
contribution of \cite{paty2020regularity} \blue{summarised in
\cref{alg:bcd_convex}, with a generalisation to convex quadratic costs $c_P(x,
y) := (x-y)^TP(x-y)$ with $P \in S_d^{++}(\R)$ (a positive-definite symmetric
matrix). We remind the notation $\Pi_{c_P}^*(g\#\mu, \nu)$ as the set of optimal
couplings between $g\#\mu$ and $\nu$ for the cost $c_P$.}

\begin{figure}[ht]
	\centering
	\begin{minipage}{.95\linewidth}
        \blue{
		\begin{algorithm}[H]
			\SetAlgoLined \KwData{Strongly convex constant $\strcvx\geq0$,
			Lipschitz constant $\gradlip \geq \strcvx$, disjoint point classes
			$I_k \subset \llbracket 1, n \rrbracket$ and discrete probability
			distributions $\mu=\sum_i a_i\delta_{x_i}$ and $\nu
			=\sum_jb_j\delta_{y_j}$.} \textbf{Initialisation:} Compute $\pi \in
			\Pi^*_{c_P}(\mu, \nu)$.\; \For{$t \in \llbracket 0, T_{\max} - 1
			\rrbracket$}{ Update $(\pot_i, g_i)_{i\in \llbracket 1, n
			\rrbracket}$ by solving the QCQP: 
            
            $\underset{\substack{\pot_1, \cdots, \pot_n \in \R\\ g_1, \cdots,
			g_n \in \R^d}}{\min} \sum_{i,j} (g_i-y_j)^TP(g_i-y_j)\pi_{i,j}$
            
            $\text{s.t.}\ \forall k \in \llbracket 1, K \rrbracket,\; \forall
            i,j \in I_k : Q(x_i, x_j, \pot_i, \pot_j, g_i, g_j) \geq 0.$
            
            Update $\pi$ by solving the discrete Kantorovich problem: $\pi \in
            \Pi_{c_P}^*(g\#\mu, \nu)$.} \caption{Alternate Minimisation for the
            Gradient of Strongly Convex Functions.} \textbf{Return}: $(\pot_i,
            g_i)_{i\in \llbracket 1, n \rrbracket}$.
			\label{alg:bcd_convex}
		\end{algorithm}}
	\end{minipage}
\end{figure}

\blue{\paragraph{Time complexity.} From a time complexity standpoint, the QCQP
problem at lines 3-4-5 bears a substantial cost. As a coarse analysis, standard
methods such as \cite{ye1989extension} have $\O(L^2N^4)$ complexity, where $N$
is the dimension of variables, here $N=(d+1)n$, and where $L = N^2 + NM + R$,
where $M$ is the number of constraints, here $M = \sum_k \#I_k^2$, and $R =
[\log|T|]$, with $T$ the sum of the non-zero integers in the float
representation of $P$ and the constraint matrix. For simplicity, we will
continue with $K=1$ and thus $M=n^2$. This yields the final (prohibitive)
complexity: $\O\left((n(d+1)+n^3+R)^2 (d+1)^4n^4\right).$ For the transport
cost, using the network simplex (see the explanation in \cite{computational_ot}
Section 3.5), omitting multiplicative logarithmic terms, the time complexity of
solving the linear Kantorovich problem between measures with $n$ and $m$ points
and cost matrix $M$ is $\O((n+m)nm\log(n+m)\log((n+m)\|M\|_{\infty}))$
\cite{tarjan1997dynamic}.

In \cref{fig:cvx_field}, we present a numerical example of the method for a map
fitting a two-dimensional standard Gaussian to a two-dimensional Gaussian
Mixture.
\begin{figure}[ht]
    \centering
    \begin{adjustbox}{valign=c}
        \begin{subfigure}[b]{0.45\textwidth}
            \centering
            \includegraphics[width=\textwidth]{
                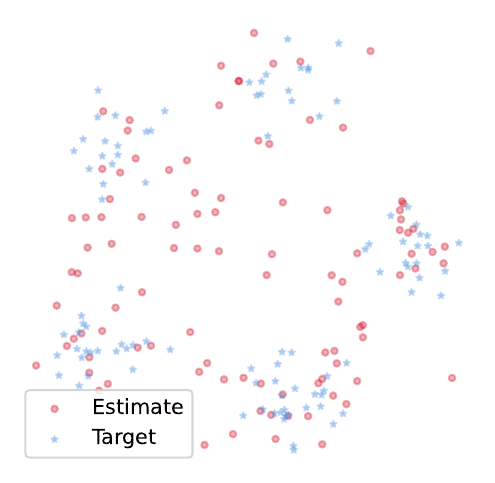} \caption{Positions of the
                optimal value positions $(g_i^*)$ for the map estimated with 
                \cref{alg:bcd_convex}.}
        \end{subfigure}
    \end{adjustbox}
    \hfill
    \begin{adjustbox}{valign=c}
        \begin{subfigure}[b]{0.45\textwidth}
            \centering
            \includegraphics[width=\textwidth]{
                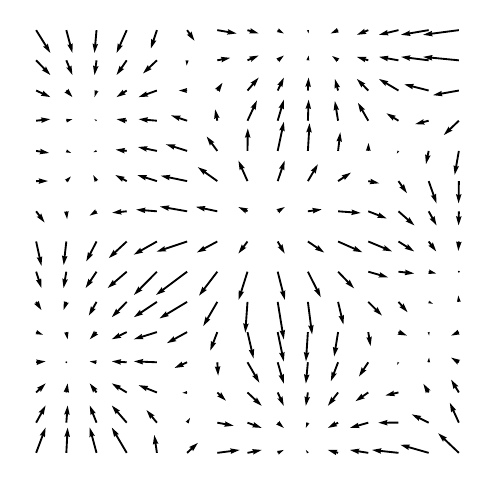} \caption{Evaluation of
                the learned map $g: \R^2 \longrightarrow \R^2$ from
                \cref{alg:bcd_convex} on a grid.}
        \end{subfigure}
    \end{adjustbox}
    \caption{\blue{Illustration of the method described in \cref{alg:bcd_convex}
    for the map problem from samples of a standard Gaussian distribution to
    samples of a Gaussian mixture. The map $g$ is constrained to be 2-Lipschitz
    and the gradient of a $1/2$-strongly convex function. The cost is taken as
    $c(x,y) = \|x-y\|_2^2$. Note that due to the constraints, we obtain an
    inexact matching, with in particular leakage between the modes of the target
    distribution.}}
    \label{fig:cvx_field}
\end{figure}
Since no public implementation of the QCQP problem from
\cite{paty2020regularity} is available, we contributed a solver for
\cref{alg:bcd_convex} for the squared-Euclidean cost in the Python OT library
\cite{flamary2021pot}, with an example
\footnote{\url{https://pythonot.github.io/auto_examples/others/plot\_SSNB.html\#sphx-glr-auto-examples-others-plot-ssnb-py}
}.}

\subsection{Numerical Method for Maps in a RKHS}\label{sec:kernels}

We introduce a relatively straightforward kernel method to solve the map problem
(\cref{eqn:CATM}). We fix a reproducing kernel Hilbert space (RKHS) $\RKHS$ of
functions $\X\longrightarrow \R^d$ of kernel $K: \X\times\X \longrightarrow
\R^{d\times d}$. We denote by $\langle \cdot, \cdot \rangle_\RKHS$ the inner
product on $\RKHS$, and $\|\cdot\|_{\RKHS}$ the associated RKHS norm on $\RKHS$.

Given discrete measures $\mu = \sum_{i=1}^na_i\delta_{x_i} \in \mathcal{P}(\X)$
and $\nu = \sum_{j=1}^mb_j\delta_{y_j}\in \mathcal{P}(\R^d)$, we will solve a
regularised variant of the map problem (\cref{eqn:CATM}):
\begin{equation}\label{eqn:kernel_problem}
    \underset{h\in \RKHS}{\argmin}\ \Tc(h\#\mu, \nu) + \lambda\|h\|_{\RKHS}^2,
\end{equation} 
for some constant $\lambda>0$ that penalises the norm of $h$, which equates to
imposing regularity on the function $h$. Given the support of $\mu$, the cost
$\Tc(h\#\mu, \nu)$ only depends on $(h(x_1), \cdots, h(x_n))$. A well known
reduction method in RKHS theory \blue{(detailed in \cref{sec:RKHS_reduction} for
completeness)} then allows to look for solutions in an $n$-dimensional linear
subspace $V$ of $\RKHS$:
\begin{equation}\label{eqn:rkhs_V}
    V := \left\{\sum_{k=1}^n K(\cdot, x_k)u_k\ :\ \forall k \in
      \llbracket 1, n \rrbracket,\; u_k \in \R^d\right\},\;  \, \, \text{ of
    }  \,\, \underset{h\in V}{\argmin}\ \Tc(h\#\mu, \nu) + \lambda\|h\|_{\RKHS}^2.
\end{equation}
Since any element $h\in V$ is characterised by its coefficients $(u_1, \cdots,
u_n) \in (\R^d)^n$, we can formulate \cref{eqn:rkhs_V} as a problem over the
$(u_i)$. First, using the kernel reproducing property, we compute
\begin{equation}\label{eqn:kernel_norm}
    \left\|\sum_{k=1}^n K(\cdot, x_k)u_k\right\|_\RKHS^2 =
    \sum_{k=1}^n\sum_{l=1}^nu_{k}^TK(x_k, x_{l})u_{l}.
\end{equation}
Concerning the transport cost term, we remind the notation for the value of the
Kantorovich discrete problem
$$\W(a, b, M) := \underset{\pi \in \Pi(\mu, \nu)}{\min}\ M\cdot \pi, $$ and in
this case, the cost matrix $M$ can be computed using the expression
\begin{equation}\label{eqn:kernel_cost_matrix}
    \forall (i,j)\in \llbracket 1, n \rrbracket \times \llbracket 1, m
    \rrbracket,\; M_{i,j} = c\left(\sum_{k=1}^n K(x_i, x_k)u_k, y_j\right).
\end{equation}
The dependency in the $(u_i)$ lies in the cost $M$. Numerically, provided that
$c$ is sufficiently regular, this allows a minimisation through classical
algorithms such as gradient descent, using differentiable implementations of the
discrete Kantorovich cost, such as \texttt{ot.emd2} \cite{flamary2021pot}. By
introducing the $nd\times nd$ matrix $\mathbf{K}$ defined by $n\times n$ blocks
$K(x_i, x_j)$ of size $d\times d$:
\begin{equation}\label{eqn:kernel_matrix}
    \mathbf{K} = \left(\begin{array}{ccc} K(x_1, x_1) & \cdots & K(x_1, x_n) \\
    \vdots & & \vdots \\
    K(x_n, x_1) & \cdots & K(x_n, x_n) \end{array}\right),
\end{equation}
and the stacked vector $\mathbf{u} \in \R^{nd}$,
\cref{eqn:kernel_norm,eqn:kernel_cost_matrix} can be re-written as matrix
products. This yields our final expression for \cref{eqn:kernel_problem}:
\begin{equation}\label{eqn:kernel_problem_alpha_i}
    \underset{\mathbf{u}\in \R^{nd}}{\min}\ \W(a, b, M(\mathbf{u})) + \lambda\mathbf{u}^T\mathbf{K}\mathbf{u},\quad M(\mathbf{u})_{i,j} := c\left(\mathbf{K}_{[i,:]}\mathbf{u}, y_j\right),
\end{equation}
where $\mathbf{K}_{[i,:]}$ denotes the sub-matrix of $\mathbf{K}$ with the $n$
lines $((i-1)d +1, \cdots, id)$, which corresponds to the $i$-th $d\times d$
block line of $\mathbf{K}$. Given optimal coefficients $\mathbf{u} = (u_1,
\cdots, u_n) \in (\R^d)^n$, a solution $h$ is defined everywhere in $\X$ using
the kernel:
$$\forall x \in \X,\; h(x) = \sum_{i=1}^n K(x, x_i)u_i. $$
\begin{remark}\label{remark:kernel_regularisation} The only constraints that are
    imposed upon a solution of \cref{eqn:kernel_problem} come from the choice of
    the kernel $K$ (or equivalently of the space $\RKHS$) and of the
    regularisation coefficient $\lambda>0$. A natural idea would be to add a
    regularisation term $R(h)$, for instance to enforce $h$ to be a gradient of
    a convex function. For \cref{lemma:kernel_reduction} to apply, one would
    need to have a regularisation which only depends on the values $(h(x_i))$,
    which is very restrictive. A possible heuristic would be to look for $h\in
    V$ regardless of this property on $R$, however the resulting problem would
    have no theoretical link to the problem over $h\in \RKHS$, unlike in our
    case. Finally, a regularisation which depends on an infinite amount of
    values $h(x)$ are numerically challenging, in the specific case of
    \textit{dense inequality constraints}, we refer to \cite{rudi2024finding} as
    a useful tool.
\end{remark}

\begin{remark}\label{remark:pertubation} A natural idea is to consider class of
    functions that are perturbations of a simple map, for instance $g = sI + h$,
    where $h$ is in a RKHS $\mathcal{H}$, and $s>0$ is a scale factor. Given
    \cref{lemma:kernel_reduction}, this tweak comes without numerical or
    theoretical cost.    
\end{remark}

We illustrate this kernel method in \cref{fig:1d_kernel_illu} with the Gaussian
kernel $K(x, y) = \exp\left(-\|x-y\|_2^2 / (2\sigma^2)\right) I_d$ and maps of
the form $g = I + h$, where $h$ is in the RKHS generated by the Gaussian kernel.

\begin{figure}[ht]
    \centering
    \begin{adjustbox}{valign=c}
        \begin{subfigure}[b]{0.45\textwidth}
            \centering
            \includegraphics[width=\textwidth]{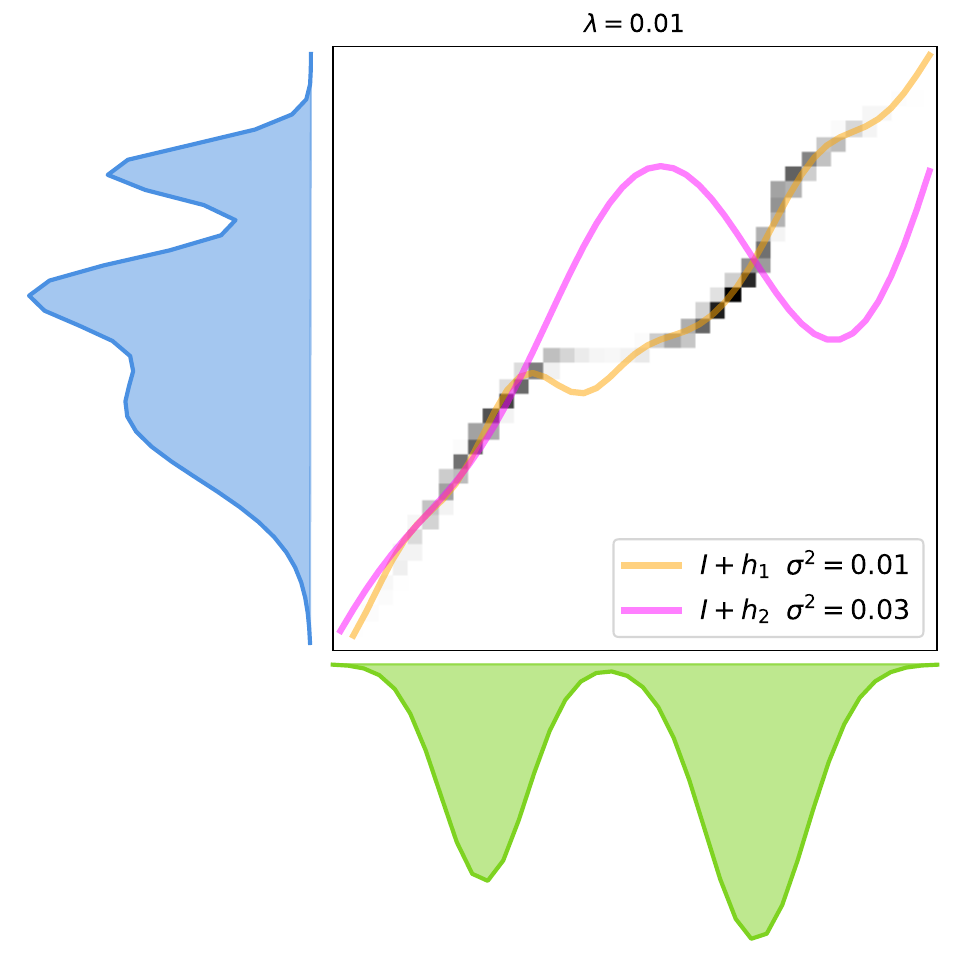}
            \caption{Kernel map solution for a regularisation $\lambda = 0.01$ 
            and multiple scales $\sigma^2$.}
        \end{subfigure}
    \end{adjustbox}
    \hfill
    \begin{adjustbox}{valign=c}
        \begin{subfigure}[b]{0.45\textwidth}
            \centering
            \includegraphics[width=\textwidth]{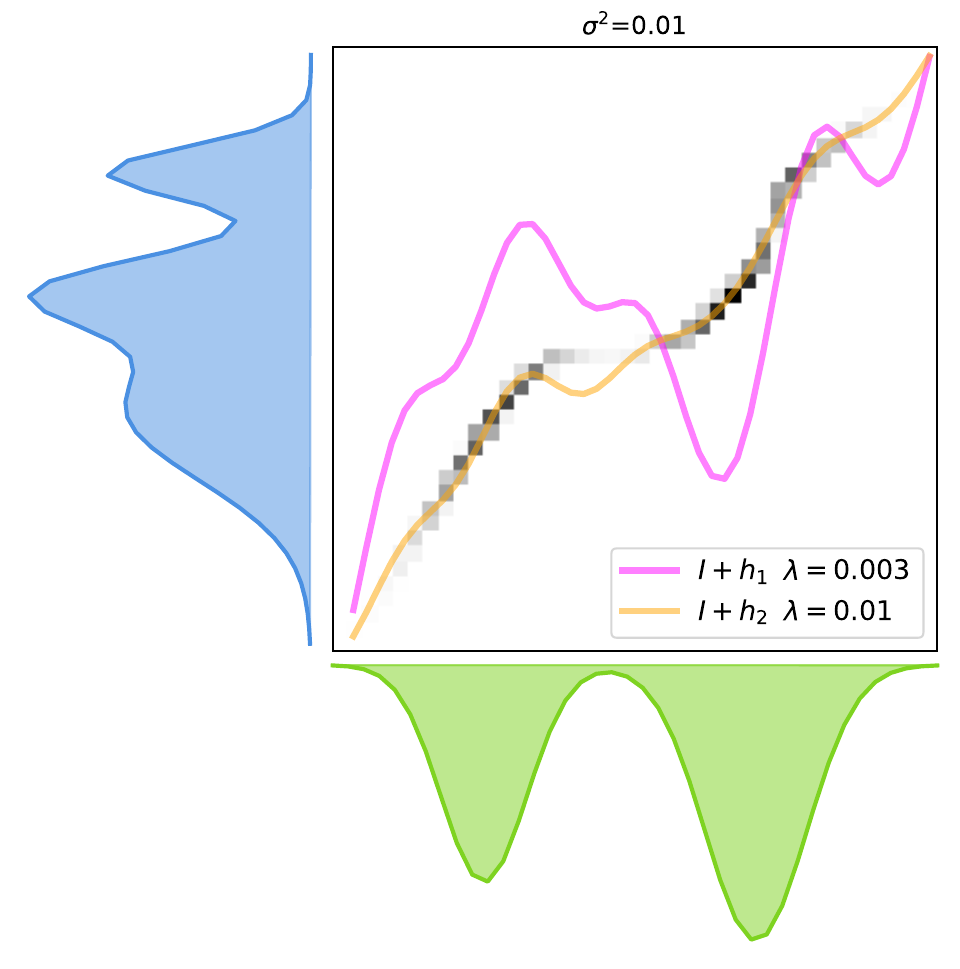}
            \caption{Kernel map solution for a scale $\sigma^2 = 0.01$ and 
            multiple regularisations $\lambda$.}
        \end{subfigure}
    \end{adjustbox}
    \caption{Illustration of the kernel method for the map problem between two
    Gaussian mixtures, using the Gaussian kernel. \blue{In greyscale, the OT
    plan is represented for reference.}}
    \label{fig:1d_kernel_illu}
\end{figure}

\blue{ From an algorithmic viewpoint, we propose in \cref{alg:kernel_GD} a
simple (sub-)gradient descent method (GD) for the discrete kernel map problem
\cref{eqn:kernel_problem_alpha_i}, and provide a convergence result in
\cref{prop:cv_kernel_GD} using results from \cref{sec:semi_algebraic} and
Ermoliev-Norkin \cite{ermoliev1997stochastic}.

\begin{figure}[ht]
	\centering
	\begin{minipage}{.95\linewidth}
        \blue{
		\begin{algorithm}[H]
			\SetAlgoLined \KwData{Gradient steps $\alpha_t > 0$, kernel
			regularisation $\lambda>0$, discrete probability distributions
			$\mu=\sum_i a_i\delta_{x_i}$ and $\nu =\sum_jb_j\delta_{y_j}$,
			kernel function $K$.} \textbf{Pre-Processing:} Compute the
			matrix $\mathbf{K}$ from \cref{eqn:kernel_matrix}\;
			\textbf{Initialisation:} Draw $\mathbf{u}_0 \in \R^{nd}$\; \For{$t
			\in \llbracket 0, T_{\max} - 1 \rrbracket$}{ $\mathbf{u}_{t+1}
			=\mathbf{u}_t - \alpha_t\left[\dr{\mathbf{u}}{}{} \left( W(a, b,
			M(\mathbf{u})) + \lambda\mathbf{u}^T\mathbf{K}\mathbf{u} \right)
			\right]_{\mathbf{u} = \mathbf{u}_t}$ (see
			\cref{eqn:kernel_problem_alpha_i}). }
			\caption{GD on the Kernel Map Parameters.}
			\label{alg:kernel_GD}
		\end{algorithm}}
	\end{minipage}
\end{figure}

Concerning time complexity, there are two main bottlenecks: the matrix-vector
computations $\mathbf{Ku}$ which incur a $\O(n^2d^2)$ cost, and solving the
discrete Kantorovich problem, which is in $\O((n+m)nm \log(n+m)
\log((n+m)\|M\|_{\infty}))$, as discussed in \cref{sec:numerics_grad_convex}.
Note that for memory efficiency, one may use map-reduce methods such as proposed
in \cite{charlier2021kernel} to avoid storing the matrix $\mathbf{K}$, at the
cost of a higher time complexity. For scalar kernels $K(x, x') = k(x, x')I_d$,
it suffices to store $\mathbf{K} := (k(x_i, x_j))_{i,j} \in \R^{n\times n}$,
reducing the memory complexity to $\O(n^2 + nd)$, and matrix-vector products to
$\O(n^2d)$.

\begin{prop}[Convergence of GD for the Kernel Method, application of
    \cite{ermoliev1997stochastic} Theorem 4.1]\label{prop:cv_kernel_GD} Take a
    locally Lipschitz and semi-algebraic (see \cref{def:semi_algebraic}) cost
    function $c$, and gradients steps $\alpha_t>0$ such that
    $\alpha_t\rightarrow0$ and $\sum_t\alpha_t=+\infty$. The iterates
    $(\mathbf{u}_t)$ of \cref{alg:kernel_GD} are such that any accumulation
    point $\mathbf{v}$ is Clarke critical: $0 \in \partial_CJ(\mathbf{v})$, with
    $J(\mathbf{u}) := W(a, b, M(\mathbf{u})) +
    \lambda\mathbf{u}^T\mathbf{K}\mathbf{u}$.
\end{prop}
\begin{proof}
    First, by \cref{prop:discrete_kanto_clarke_semi_algebraic} and by
    semi-algebraicity and local Lipschitzness of $c$, $J$ is locally Lipschitz
    and semi-algebraic. Using now \cref{lemma:semi_algebraic_props}, we have the
    sufficient regularity conditions to use the convergence result of
    \cite{ermoliev1997stochastic} (Theorem 4.1).
\end{proof}
} \blue{\subsection{Gradient Descent for Neural Networks}\label{sec:sgd_nn_map}}
\blue{ We now consider the case where the function class $G$ is the class of
neural networks introduced in \cref{eqn:def_NN}, with parameters in a compact
set $\Theta \subset \R^p$, which we also assume to be convex. We will introduce
a technical modification of the neural network from \cref{eqn:def_NN} and
consider the parametrised function:
\begin{equation}\label{eqn:def_HNN}
    h :=(\theta, x) \longmapsto g_{P_\Theta(\theta)}(x),
\end{equation}
with $g$ the map defined \cref{eqn:def_NN}, and where the map $P_\Theta : \R^p
\longrightarrow \Theta$ denotes the orthogonal projection onto $\Theta$. This
re-writing allows us to define the NN $h$ on all of $\R^p \supset \Theta$. In
practice, SGD with this network is essentially equivalent to projecting the
parameters after each gradient step (with the technicality that in our
formalism, the gradient of $P_\Theta$ is included in the backpropagation). 

To solve the map problem of minimising $\T_c(h(\theta, \cdot)\#\mu, \nu)$ in
practice, we consider a commonly used minibatch surrogate loss $F(\theta)$,
which we define in \cref{eqn:def_F}. Given a dataset $X^{(n)} \in \R^{n\times k}
= (x_1, \cdots, x_n)$, we will denote abusively $h(\theta, X^{(n)}) \in
\R^{n\times d} := (h(\theta, x_1), \cdots, h(\theta, x_n))$. Given a dataset
$X^{(n)} \in \R^{n\times k}$, the measure $\delta_{X^{n}} \in \mathcal{P}(\R^k)$
denotes $\frac{1}{n}\sum_i\delta_{x_i}$. Similarly, we will denote a target
dataset $Y^{(m)}$. The loss $F$ we consider is
\begin{equation}\label{eqn:def_F}
    F(\theta) := \int \T_c(\delta_{h(\theta, X^{(n)})}, \delta_{Y^{(m)}})
    \dd\mu^{\otimes n}(X^{(n)})\dd\nu^{\otimes m}(Y^{(m)}).
\end{equation}
The loss $F$ will be minimised by Stochastic Gradient Descent over $\theta$,
where the stochasticity is on the data batches $X^{(n)}$ and $Y^{(m)}$, as
described in \cref{alg:SGD}.

\begin{figure}[ht]
	\centering
	\begin{minipage}{.95\linewidth}
        \blue{
		\begin{algorithm}[H]
			\SetAlgoLined \KwData{Gradient steps $\alpha_t > 0$, probability
            distributions $\mu \in \mathcal{P}(\R^k)$ and $\nu \in
            \mathcal{P}(\R^d)$, NN $h(\theta,.)$.}
            \textbf{Initialisation:} Draw $\theta_0 \in \Theta$\; \For{$t \in
            \llbracket 0, T_{\max} - 1 \rrbracket$}{ Draw $X^{(n)} \sim
            \mu^{\otimes n},\;  Y^{(m)} \sim \nu^{\otimes m}$. \\ 
            $\theta_{t+1} = \theta_t - \alpha_t\left[\dr{\theta}{}{}
            \T_c(\delta_{h(\theta, X^{(n)})},\delta_{Y^{(m)}}) \right]_{\theta =
            \theta_t}$. }
			\caption{Training a NN map for the cost $\T_c$.}
			\label{alg:SGD}
		\end{algorithm}}
	\end{minipage}
\end{figure}

An important remark is that this formalism bears strong similarities to the
alternate minimisation framework studied in \cref{sec:alternate_minimisation}
for the squared-Euclidean cost. Indeed, it can be seen as an alternation of the
map parameters $\theta$ and the (minibatch) OT plan $\pi$ in $\T_c$ (line 4):
the optimisation over $\pi$ is done by solving the linear program when computing
the cost $\T_c$, and then one gradient step of optimisation over $\theta$ is
performed. To study \cref{alg:SGD} theoretically, we will give precise meaning
to the partial derivative at line 4 using the notions introduced in
\cref{sec:semi_algebraic}. Numerically, the sub-gradients in question are
computed by automatic differentiation. Note that $P_\Theta$ is Lipschitz and
semi-algebraic. 

Thanks to the regularity result on the OT cost proved in
\cref{prop:discrete_kanto_clarke_semi_algebraic}, we can use recent SGD
convergence results by Bolte, Le and Pauwels \cite{bolte2023subgradient} to show
that the iterates of \cref{alg:SGD} converge in a certain sense. First, to give
sense to the gradient in \cref{alg:SGD}, we remark that for locally Lipschitz
semi-algebraic activation functions, the map $h(\cdot, \cdot)$ is semi-algebraic
and locally Lipschitz. By composition using
\cref{prop:discrete_kanto_clarke_semi_algebraic}, the sample loss function: 
$$f(\cdot, X^{(n)}, Y^{(m)}) := \theta \longmapsto \T_c(\delta_{h(\theta,
X^{(n)})}, \delta_{(Y^{(m)})}),$$ is locally Lipschitz and semi-algebraic. We
can select a semi-algebraic sub-gradient $\varphi: \R^p \times \R^{n\times k}
\times \R^{m\times d} \longrightarrow \R^p$ such that 
$$\forall \theta\in \R^p,\; \forall X^{(n)} \in \R^{n\times k},\; \forall
Y^{(m)} \in \R^{m\times d},\; \varphi(\theta, X^{(n)}, Y^{(m)}) \in \partial_C
f(\theta, X^{(n)}, Y^{(m)}),$$ where the selection can be done by lexicographic
order on coordinates, for example. Note that $f(\cdot, X^{(n)}, Y^{(m)})$ is
differentiable almost-everywhere, and that at differentiable points, $\varphi$
equates its usual gradient. The choice of sub-gradient performed by automatic
differentiation satisfies this condition (see a discussion on this procedure in
\cite{bolte2021conservative,davis2020stochastic}.) We remind that the population
loss function is $F = \theta \longmapsto \int f(\theta, X^{(n)},
Y^{(m)})\dd\mu^{\otimes n}(X^{(n)})\dd\nu^{\otimes m}(Y^{(m)})$ in this setting.

\begin{prop}[Convergence of SGD for NN maps, application of
    \cite{bolte2023subgradient} Theorem 3]\label{prop:map_nn_sgd_converges}
    Assume that $\mu, \nu$ are discrete measures or compactly supported measures
    with semi-algebraic densities with respect to the Lebesgue measure. Assume
    that the NN $h$ is defined as in \cref{eqn:def_HNN}, with locally Lipschitz
    semi-algebraic activation functions. Assume that $\Theta$ is compact, convex
    and semi-algebraic. Suppose that the cost function $c: \R^d\times \R^d
    \longrightarrow \R_+$ is locally Lipschitz and semi-algebraic. Take gradient
    steps $(\alpha_t)_{t\in \N} \in (0, +\infty)^\N$ such that $\sum_t \alpha_t
    = +\infty$ with $\alpha_t = o(1/\log(t))$.

    Then there exists a set of possible steps $A\subset (0, +\infty)$ whose
    complement is finite, and a set of possible initialisations $\Theta_0
    \subset \Theta$ of full measure, such that for each step sequence
    $(\alpha_t) \in A^\N$ verifying the conditions, the stochastic gradient
    descent iterates:
    $$\theta_0 \in \Theta_0,\; \forall t\in \N,\; \theta_{t+1} = \theta_t -
    \alpha_t\varphi(\theta_t, X_t^{(n)}, Y_t^{(m)}),\; X_t^{(n)}\sim\mu^{\otimes
    n},\; Y_t^{(m)}\sim \nu^{\otimes m},
    $$
    verify that almost-surely, $(F(\theta_t))$ converges, and almost-surely 
    any accumulation point $\oll{\theta}$ of $(\theta_t)$ is such that $0 \in 
    \partial_C F(\oll{\theta})$, under the (mild) additional assumption that that the trajectories $(\theta_t)$ are almost-surely bounded.
\end{prop}
\begin{proof}
    We apply Bolte-Le-Pauwels \cite{bolte2023subgradient}, 
    Theorem 3 to the NN $h$ with the discrete OT loss from
    \cref{prop:discrete_kanto_clarke_semi_algebraic}. Thanks to the assumptions 
    formulated in the result statement, to 
    \cref{prop:discrete_kanto_clarke_semi_algebraic} and to the construction of 
    a semi-algebraic sub-gradient selection $\varphi$, 
    we have collected all the conditions for
    Bolte-Le-Pauwels, Theorem 3, yielding the result. For the case where one or
    both of $\{\mu, \nu\}$ is/are discrete, we applied their Remark 3.
\end{proof}

In \cref{fig:nn_field}, we present a numerical example of the method for a map
fitting a two-dimensional standard Gaussian to a two-dimensional Gaussian
Mixture.

\begin{figure}[ht]
    \centering
    \begin{adjustbox}{valign=c}
        \begin{subfigure}[b]{0.45\textwidth}
            \centering
            \includegraphics[width=\textwidth]{
                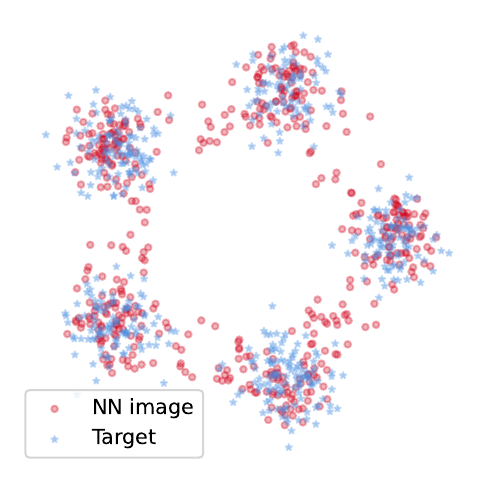}
            \caption{Comparison of images of samples of the source Gaussian by 
            the learned NN map to the target Gaussian Mixture.}
        \end{subfigure}
    \end{adjustbox}
    \hfill
    \begin{adjustbox}{valign=c}
        \begin{subfigure}[b]{0.45\textwidth}
            \centering
            \includegraphics[width=\textwidth]{
                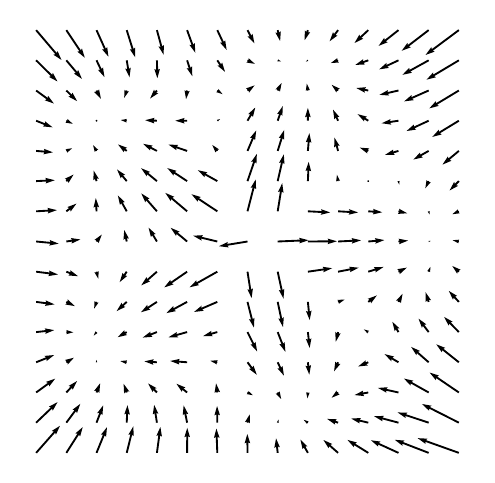}
            \caption{Evaluation of the learned NN map $g: \R^2 
            \longrightarrow \R^2$ on a grid.}
        \end{subfigure}
    \end{adjustbox}
    \caption{\blue{Illustration of the method described in~\cref{alg:SGD} for 
    the map problem from samples
    of a standard Gaussian distribution to samples of a Gaussian mixture. The
    map $g$ is of the form $g = I + h$, where $h$ is a small 4-layer NN with 
    ReLU activation functions, and weights constrained to $[-1/2, 1/2]$. The 
    cost is taken as $c(x,y) = \|x-y\|_2^2$. Note that due to the (indirect) 
    constraint on the Lipschitz constant, we obtain an inexact matching, 
    with in particular leakage between the modes of the target distribution.}}
    \label{fig:nn_field}
\end{figure}
}
\blue{\subsection{Illustrative Application to Colour Transfer}\label{sec:colour_transfer}}
\blue{
In this section, we consider the task of colour transfer, which consists in 
transforming the colour distribution of a source image onto the colour 
distribution of a target image. An $(n\times m)$ RGB image is seen as a 
3-tensor $\I \in [0, 1]^{n\times m\times 3}$, and its colour distribution is 
then a discrete measure $\mu = \frac{1}{nm}\sum_{i,j}\delta_{\I_{i,j}}$ on 
$\R^3$.

We illustrate that a learned map $g: \R^3 \longrightarrow \R^3$ which is 
optimised to transfer the colours of a source image $\I_s$ onto a target image 
$\I_t$ can be used on a new image $\I$ to transfer its colours. This
is made possible since the map $g$ is defined everywhere, and not only
at the points of $\mu$. We consider the cost $c(x, y) = \|x-y\|_2^2$, and a 
simple NN map $g$ using \cref{alg:SGD} on the 
source and target images, and then apply the map to new images. We present the 
results in \cref{fig:colour_transfer} for three different training tasks. 
Notice how the constraint on the map $g$ allows us to have a
colour transfer that is robust to outliers. In 
\cref{fig:colour_transfer_rgb} in \cref{sec:colour_transfer_rgb}, we present 
the results in the RGB space, seeing the images as pixel point clouds. 

\begin{figure}[H]
    \begin{center}
        \includegraphics[width=0.7\textwidth]{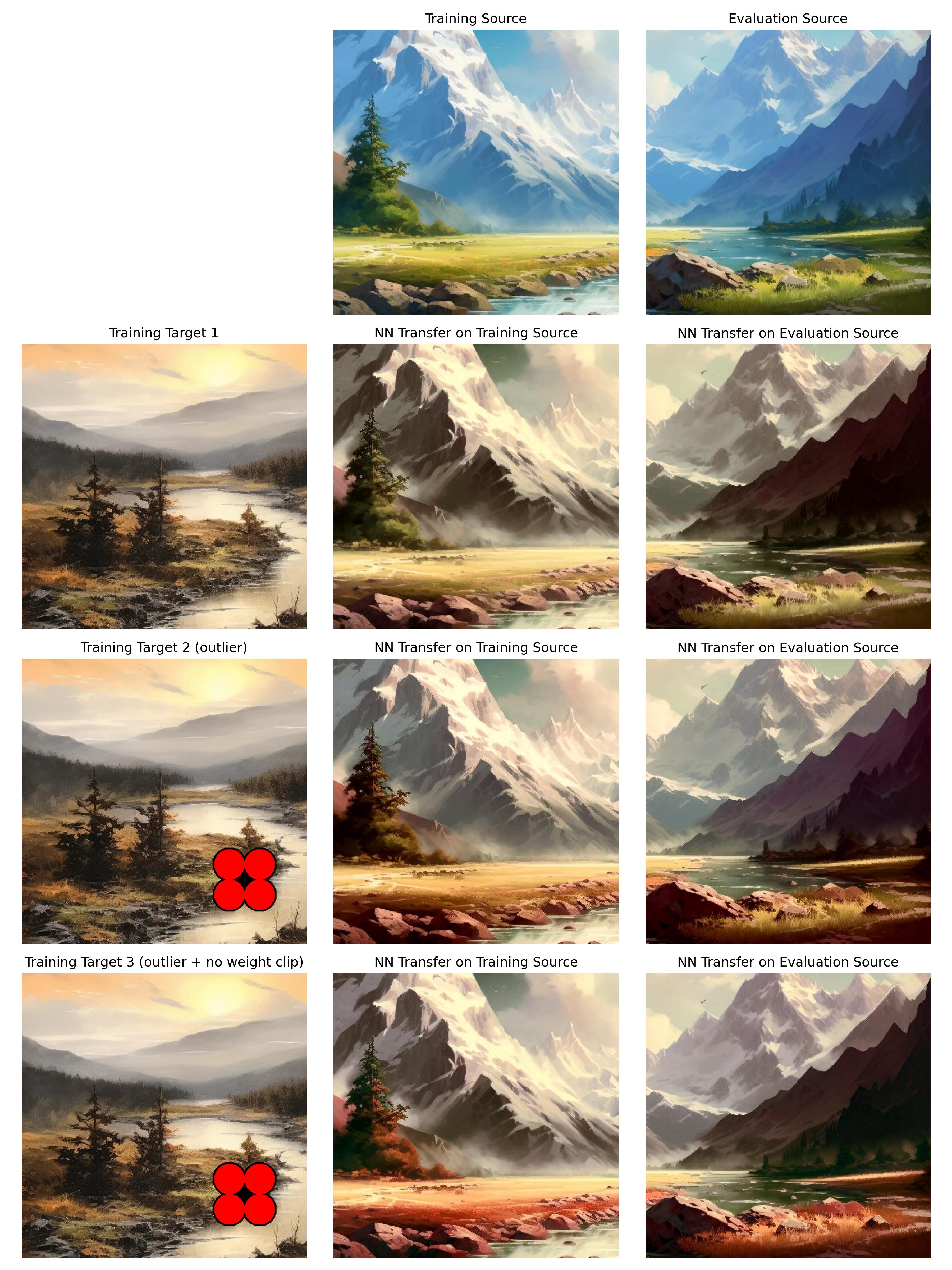}
    \end{center}
    \caption{\blue{Colour transfer using a NN map trained on the task of
    transferring the colour distribution of the Training Source image onto 
    three different Training Target images (column 1, rows 2 to 4). In the 
    second training (row 3), the target image presents an outlier in 
    the colour distribution. In the third training (row 4), 
    the same target image with outliers is used, this 
    time with no weight clipping. For all trainings,
    the learned map is applied to the training image (second column) and 
    to the test image "Evaluation Source" (column 3).}}
    \label{fig:colour_transfer}
\end{figure}}
\section{Conclusion and Outlook}

In this paper, we have considered the problem of finding an optimal transport
map $g$ between two probability measures $\mu$ and $\nu$ under the constraint
that $g\in G$, where $G$ is a given set of functions ($L$-Lipschitz, gradient of
a convex function,  for instance). We have given general assumptions to ensure
the existence of an optimal map $g$, we have studied the relationship between
our problem and many other concepts in Optimal Transport, and also the link with
kernel methods. We have also explained how to solve the problem from a practical
point a view \blue{with convergence guarantees, and an application to colour
transfer}.

We believe that there are two important but difficult questions that should be
investigated in future work. The first is the question of the uniqueness of an
optimal map. We have given a partial answer to this question, but it seems to be
a difficult question in its whole generality. Having a result of uniqueness
would then open the way to new questions, such as the use of $g$ to compare
measures in a way similar to Linearised Optimal Transport, or the study of the
statistical properties of $g$ (related to the sample complexity). The second
important question is the addition of constraints in the kernel method, more
precisely: how to translate a set of functions $G$ (like the set of gradients of
convex functions for instance) into a RKHS representation? 

\subsection*{Acknowledgements}

We extend our warmest thanks to Nathaël Gozlan for his valuable input regarding
technical assumptions for the existence result. We thank Joan Glaunès
for the fruitful time we spent working 
together on kernel problems. \blue{We also want to thank Tam Le for providing
references and insight for non-smooth and non-convex optimisation questions. We
are grateful for the remarks and feedback of two anonymous reviewers, which in
particular allowed significant improvement of the map existence results.}

This research was funded in part by the Agence nationale de la
recherche (ANR), Grant ANR-23-CE40-0017 and by the
France 2030 program, with the reference ANR-23-PEIA-0004.

\bibliography{ecl}
\bibliographystyle{plain}

\appendix
\section{Appendix}

\subsection{Continuous-to-Discrete Case: Semi-discrete OT}\label{sec:continuous_to_discrete}

In the alternate optimisation scheme proposed in
\cref{sec:alternate_minimisation}, the step with $g$ fixed can be seen as
semi-discrete Optimal Transport, whenever the target measure $\nu$ is discrete,
and when the measure $g\#\mu$ is absolutely continuous with respect to the
Lebesgue measure. The condition $g\#\mu\ll \Leb$ arises naturally whenever the
source measure $\mu$ is itself absolutely continuous, which we will assume for
this section.

Specifically, the sub-problem of computing
$$\Tc(g\#\mu, \nu)$$ can be seen as a semi-discrete optimal transport problem
between $g\#\mu$ and $\nu$ (see \cite{merigot2020ot_course} for a course with a
detailed section on semi-discrete OT). To apply semi-discrete optimal transport
methods to this sub-problem, we need to verify $g\#\mu \ll \Leb$. First, it
follows from the definition that if $g\#\Leb \ll \Leb$, then, since we assume
$\mu$ is absolutely continuous, $g\#\mu \ll \Leb$ would follow. In
\cref{lemma:image_measure_AC}, we provide relatively general sufficient
conditions on the map $g: \R^d\longrightarrow \R^d$.

\begin{lemma}\label{lemma:image_measure_AC} Let $g: \R^d \longrightarrow \R^d$
    locally Lipschitz such that for $\Leb$-a.e. $x\in \R^d,\; \det\partial
    g(x) \neq 0$. Then $g\#\Leb \ll \Leb$.
\end{lemma}

\begin{remark}
    By Rademacher's theorem (\cite{evans2018measure},  Theorem 3.2), a locally
    Lipschitz function is differentiable $\Leb$-a.e..
\end{remark}

\begin{proof}
    First, we remind that $J_g := x \longmapsto |\det \partial g(x)|$ (defined
    $\Leb$-almost-everywhere) is locally integrable since $g$ is locally
    Lipschitz. We now prove that $g\#\Leb \ll \Leb$ by considering the
    intersection of compact sets and $\Leb$-null sets. Let $\mathcal{K}\subset
    \R^d$ a compact set and $E \subset \R^d$ a Borel set such that $\Leb(E)=0$.
    By the area formula (\cite{evans2018measure}, Theorem 3.8), the following
    equality holds
    \begin{equation}\label{eqn:area_formula_application}
        \int_{g^{-1}(E)\cap \mathcal{K}} J_g(x)\dd x =
        \int_{\R^d}\mathcal{H}^0(g^{-1}(\{y\})\cap \mathcal{K}\cap
        g^{-1}(E))\dd y = \int_{E}\mathcal{H}^0(g^{-1}(\{y\})\cap \mathcal{K})\dd y,
    \end{equation}
    where $\mathcal{H}^0$ denotes the 0-dimensional Hausdorff measure (the
    counting measure). The left-side expression in
    \cref{eqn:area_formula_application} is finite. Since $\Leb(E)=0$, it follows
    that the right-most term in \cref{eqn:area_formula_application} is 0, thus
    $$\int_{g^{-1}(E)\cap \mathcal{K}} J_g(x)\dd x = 0. $$ Since by assumption
    $J_g$ is positive almost-everywhere, it follows that $\Leb(g^{-1}(E)\cap
    \mathcal{K})=0$. Since the compact set $\mathcal{K}$ was chosen arbitrarily,
    we conclude that $\Leb(g^{-1}(E))=0$, which shows $g\#\Leb\ll
    \Leb$.
\end{proof}

\subsection{Lemmas on Pseudo-inverses and Quantile Functions}
\label{appendix:quantile_inverses}

To begin with, we introduce some notions regarding pseudo-inverses of
non-decreasing functions. 

\begin{definition}
	For $\psi:\R\longrightarrow\R$ non-decreasing, its \textbf{right-inverse} is
	defined as the function:
	$$\rinv{\psi}: \R \longrightarrow \oll{\R}: \quad\forall p \in \R,\;
	\rinv{\psi}(p) := \inf\left.\left\{x \in \R\  \right|\ \psi(x)\geq
	p\right\}.$$ For $\phi:\R\longrightarrow\R$ non-decreasing, its
	\textbf{left-inverse} is defined as the function:
	$$\linv{\phi}: \R \longrightarrow \oll{\R}: \quad \forall p \in \R,\;
	\linv{\phi}(p) := \sup\left.\left\{x \in \R\  \right|\ \phi(x)\leq
	p\right\}.$$
\end{definition}

These notions are particularly useful for the definition of the right-inverse of
the cumulative distribution function of a probability measure $\mu$: $F_\mu:= x
\longmapsto \mu((-\infty, x])$, and for the left-inverse of the function $G_\mu
:= x \longmapsto \mu((-\infty, x))$. We recall and prove some well-known
properties of pseudo-inverses (see \cite{embrechts2013note} for a detailed
presentation of right-inverses). For a non-decreasing function $\psi$, we
define $\psi(-\infty) := \lim_{x\searrow -\infty}\psi(x)\in \R\cup\{-\infty\}$
and $\psi(+\infty) := \lim_{x\nearrow +\infty}\psi(x)\in \R\cup\{+\infty\}$.

\begin{lemma}\label{lemma:pseudo-inverses}
	\begin{enumerate}
		\item Let $\psi: \R \longrightarrow \R$ non-decreasing and
		right-continuous. Then:
		\begin{enumerate}
			\item For all $(x,p) \in \R^2,\; \psi(x)\geq p \Longleftrightarrow x
			\geq \rinv{\psi}(p)$.
			\item If $\rinv{\psi}(p) < +\infty,\; \psi(\rinv{\psi}(p))\geq p.$
		\end{enumerate}
		\item Let $\phi: \R \longrightarrow \R$ non-decreasing and
		left-continuous. Then:
		\begin{enumerate}
			\item For all $(x,p) \in \R^2,\; \phi(x)\leq p \Longleftrightarrow x
			\leq \linv{\phi}(p)$.
			\item If $\linv{\phi}(p) > -\infty,\; \phi(\linv{\phi}(p)) \leq p.$
		\end{enumerate}
		\item Under the assumptions above, if additionally $\phi\leq \psi$, then
		$\linv{\phi} \geq \rinv{\psi}$.
	\end{enumerate}
\end{lemma}

\begin{proof}
	We detail the proofs for claims 1.(a) and 1.(b), the arguments for 2.(a) and
	2.(b) being essentially the same. First, we let $p\in \R$ such that
	$\rinv{\psi}(p) < +\infty$, which is equivalent to supposing $A_p \neq
	\varnothing$, with $A_p := \left.\left\{x \in \R\  \right|\ \psi(x)\geq
	p\right\}.$ We also suppose $\rinv{\psi}(p)>-\infty$, which is equivalent to
	assuming that $A_p$ is lower-bounded. Since $A_p \neq \varnothing$, we can
	choose a decreasing sequence $(x_n)\in A_p^\N$ such that $x_n \xrightarrow[n
	\longrightarrow +\infty]{}\rinv{\psi}(p)$. Since $\psi$ is right-continuous
	and $\rinv{\psi}(p)\in \R$, we have $\psi(x_n)\xrightarrow[n \longrightarrow
	+\infty]{}\psi(\rinv{\psi}(p))$. However, since each $x_n \in A_p$, we have
	$\psi(x_n)\geq p$, and by taking the limit in the inequality we deduce
	$\psi(\rinv{\psi}(p))\geq p$. If $\rinv{\psi}(p)=-\infty$, then the same
	argument with $x_n \xrightarrow[n \longrightarrow +\infty]{}-\infty$ and
	$\psi(-\infty) := \lim_{x\searrow -\infty}\psi(x)\in \R\cup\{-\infty\}$ also
	shows $\psi(\rinv{\psi}(p))\geq p$, which concludes the proof of 1.(b).

	For 1.(a), we first assume $\rinv{\psi}(p) < +\infty$. In this case, by 1b)
	we have $\phi(\rinv{\psi}(p))\geq p$, thus $[\rinv{\psi}(p), +\infty)
	\subset A_p$. Yet by definition of $\rinv{\psi}(p)$, $x\in A_p
	\Longrightarrow x \geq \rinv{\psi}(p)$, thus we conclude $A_p =
	[\rinv{\psi}(p), +\infty)$, which is exactly the same statement as
	$\psi(x)\geq p \Longleftrightarrow x\geq \rinv{\psi}(p)$. If $\rinv{\psi}(p)
	= +\infty$, then the equivalence still holds, since $\psi(x)\geq p
	\Longleftrightarrow x\in A_p$, with $A_p=\varnothing$.

	Regarding 3., let $p \in \R$ such that $\rinv{\phi}(p)>-\infty$. Then $\{ x
	\in \R\ |\ \phi(x)\leq p\} = (-\infty, \rinv{\phi}(p)]$ by 2.a), thus
	$\rinv{\phi}(p) = \inf\{x \in \R\ |\ \phi(x)>p\}$. The previous equality
	also holds when $\rinv{\phi}(p)=-\infty$. Now since $\phi \leq \psi$, we
	have $\{x\in \R\ |\ \phi(x)>p\} \subset \{x\in \R\ |\ \psi(x)\geq p\}$, and
	taking the infimum yields $\linv{\phi}(p)\geq \rinv{\psi}(p)$.
\end{proof}

Using this result, we can now prove \cref{lemma:push_forward_increasing_quantile}.

\begin{proof}[Proof of \cref{lemma:push_forward_increasing_quantile} ]
	First, notice that as a cumulative distribution function, $F_\mu$ is
	non-decreasing and right-continuous. Since $g$ is non-decreasing, we have
	for $p\in (0, 1):$
	$$F_{g\#\mu}(g\circ\rinv{F_\mu}(p)) = \P_{X\sim \mu}\left(g(X)\leq g\circ
	\rinv{F_\mu}(p)\right) \geq \P_{X\sim\mu}(X\leq \rinv{F_\mu}(p)) =
	F_\mu\circ\rinv{F_\mu}(p).$$ Now if $\rinv{F_\mu}(p) < +\infty$, we have
	$F_\mu\circ\rinv{F_\mu}(p) \geq p$ by \cref{lemma:pseudo-inverses} 1.b). We
	now turn to the case $\rinv{F}_\mu(p)=+\infty$, which implies that $\forall
	x \in \R,\; F_\mu(x)<p$. Since $F_\mu$ is a cumulative distribution
	function, this implies $p\geq1$, which we excluded. We have shown that
	$F_{g\#\mu}(g\circ\rinv{F_\mu}(p)) \geq p$, thus by definition of
	$\rinv{F_{g\#\mu}}(p)$, we have $\rinv{F_{g\#\mu}}(p) \leq
	g\circ\rinv{F_\mu}(p)$.

	Regarding the converse inequality, we will show that the set $N :=
	\left\lbrace p \in (0,1)\ :\ \rinv{F_{g\#\mu}}(p) < g\circ
	\rinv{F_\mu}(p)\right\rbrace$ is Lebesgue-null. Let $p\in N$ and $\alpha \in
	\left[\rinv{F_{g\#\mu}}(p), g\circ\rinv{F_{g\#\mu}}(p)\right)$. As done
	earlier with $F_\mu$, using \cref{lemma:pseudo-inverses} and the fact that
	$F_{g\#\mu}$ is a c.d.f. and that $p<1$, we have
	$F_{g\#\mu}\circ\rinv{F_{g\#\mu}}(p)\geq p$. Since $F_{g\#\mu}$ is
	non-decreasing, we obtain $p \leq F_{g\#\mu}(\alpha)$. We re-write
	$F_{g\#\mu}(\alpha)$ using its definition, then use the fact that $g$ is
	non-decreasing:
	\begin{equation}\label{eqn:lemma_push_forward_quantile_p}
		p \leq F_{g\#\mu}(\alpha) = \P_{X\sim\mu}\left(g(X)\leq \alpha\right) \leq \P_{X\sim\mu}\left(g(X)<g\circ\rinv{F_\mu}(p)\right) \leq \P_{X\sim\mu}\left(X<\rinv{F_\mu}(p)\right) =: G_\mu(\rinv{F_\mu}(p)).
	\end{equation}
	We now want to show that $G_\mu(\rinv{F_\mu}(p))\leq p$. Since $G_\mu \leq
	F_\mu$ and since they are non-decreasing and $G_\mu$ is left-continuous, and
	$F_\mu$ is right-continuous (by the axiomatic properties of $\mu$), by
	\cref{lemma:pseudo-inverses} item 3, we have $\linv{G_\mu} \geq
	\rinv{F_\mu}.$ In particular, since $G_\mu$ is non-decreasing, we have
	$$G_\mu(\rinv{F}_\mu(p))\leq G_\mu(\linv{G_\mu}(p)) \leq p,$$ where the
	final inequality comes from \cref{lemma:pseudo-inverses} item 2b), with
	$\linv{\phi}(p)>-\infty$ since we chose $p>0$.

	We have shown that $G_\mu(\rinv{F_\mu}(p)) \leq p$, thus every equality in
	\cref{eqn:lemma_push_forward_quantile_p} is an equality, and as a result,
	for any $\alpha\in \left[\rinv{F_{g\#\mu}}(p),
	g\circ\rinv{F_{g\#\mu}}(p)\right)$, we have $F_{g\#\mu}(\alpha)=p$, thus the
	right-inverse $\rinv{F_{g\#\mu}}$ has a jump-discontinuity at $p$: 
	$$\underset{q<p}{\sup}\ \rinv{F_{g\#\mu}}(q) = \rinv{F_{g\#\mu}}(p) <
	\underset{p<q}{\inf}\ \rinv{F_{g\#\mu}}(q).$$ We conclude that $N$ is a
	subset of the set $J$ of jump-discontinuities of $\rinv{F_{g\#\mu}}$, and
	since $\rinv{F_{g\#\mu}}$ is non-decreasing, $J$ is countable and thus of
	Lebesgue measure 0. As a result, we have for almost-every $p\in [0, 1],\;
	\rinv{F_{g\#\mu}}(p) = g\circ \rinv{F}_\mu(p)$.
\end{proof}

\blue{\subsection{Reminder on Reduction in RKHS
methods}\label{sec:RKHS_reduction}}

The reduction method in RKHS is known since \cite{aronszajn1950theory} (Section
3), but given the simplicity of the arguments and for the sake of
self-completeness, we provide a proof and presentation in
\cref{lemma:kernel_reduction}.

\begin{lemma}\label{lemma:kernel_reduction} Consider a cost function $J: \RKHS
    \longrightarrow \R_+$ which can be written as $J(h) = J\left(h(x_1),\cdots,
    h(x_n)\right)$, then if $h^*\in \RKHS$ is a solution of
    $$\underset{h\in \RKHS}{\argmin}\ J(h) + \lambda \|h\|_{\RKHS}^2,$$ then
    $h_V$, the orthogonal projection of $h^*$ onto $V$ (defined in
    \cref{eqn:rkhs_V}) verifies :
    $$\forall i \in \llbracket 1, n \rrbracket,\; h_V(x_i) = h^*(x_i),$$ and as
    a result $J(h_W) = J(h^*)$, which leads to the following problem reduction:
    \begin{equation}\label{eqn:kernel_reduction}
        \underset{h\in \RKHS}{\argmin}\ J(h) + \lambda \|h\|_{\RKHS}^2 = \underset{h\in V}{\argmin}\ J(h) + \lambda \|h\|_{\RKHS}^2.
    \end{equation}
\end{lemma}
\begin{proof}
    To show that $\forall i \in \llbracket 1, n \rrbracket,\; h_V(x_i) =
    h^*(x_i)$, we will show that
    $$V^\perp = H_0 := \left.\left\lbrace h\in \RKHS\ \right|\ \forall i \in
    \llbracket 1, n \rrbracket,\; h(x_i) = 0 \right\rbrace .$$ Indeed, 
    \begin{align*}
        h \in H_0 &\Longleftrightarrow \forall i \in \llbracket 1, n \rrbracket,\; g(x_i) = 0\\
        &\Longleftrightarrow \forall i \in \llbracket 1, n \rrbracket,\; \forall u \in \R^d,\; g(x_i) \cdot u = 0\\
        &\Longleftrightarrow \forall i \in \llbracket 1, n \rrbracket,\; \forall u \in \R^d,\; \delta_{x_i}^u g = 0\\
        &\Longleftrightarrow \forall i \in \llbracket 1, n \rrbracket,\; \forall u \in \R^d,\; \langle g, K(\cdot, x_i)u \rangle_\RKHS = 0\\
        &\Longleftrightarrow f \in V^\perp,
    \end{align*}
    where $\delta_x^u$ is the linear form $h \longmapsto h(x) \cdot u$, whose
    Riesz representation in $\RKHS$ is $K(\cdot, x)u$ by the kernel reproducing
    property. We conclude the proof with the fact that as an orthogonal
    projection, $\|h_V\|_\RKHS^2 \leq \|h^*\|_\RKHS^2$, which shows that the
    cost of $h_V$ is less than the cost of $h^*$.
\end{proof}

\blue{\subsection{Colour Transfer: RGB Point Cloud Viewpoint}\label{sec:colour_transfer_rgb}

In \cref{fig:colour_transfer_rgb}, we provide a visualisation of the colour
transfer from \cref{fig:colour_transfer} in the RGB space.}

\begin{figure}[ht]
    \begin{center}
        \includegraphics[width=.8\textwidth]{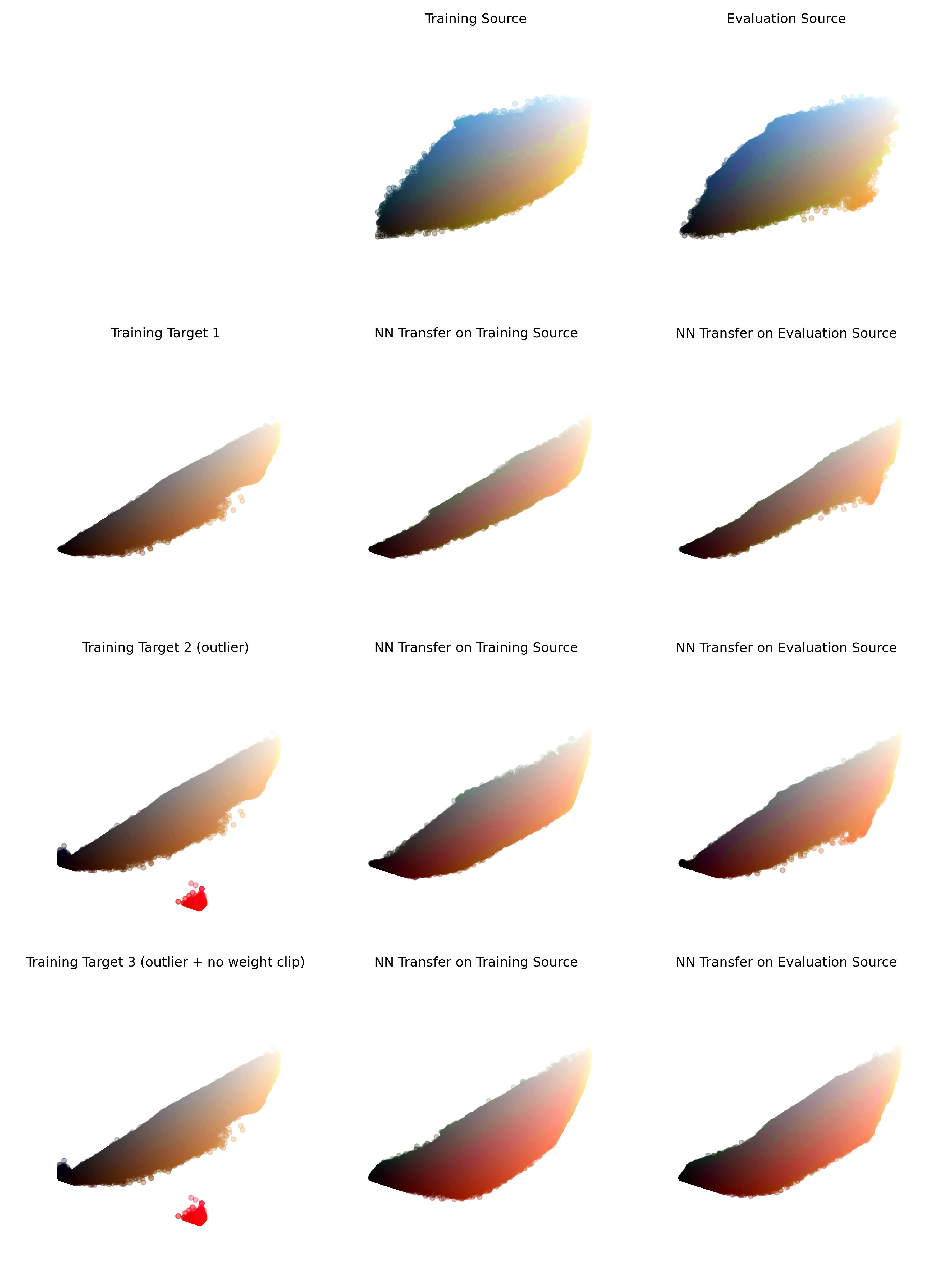}
    \end{center}
    \caption{\blue{RGB space visualisation of the colour transfer from 
    \cref{fig:colour_transfer}.}}
    \label{fig:colour_transfer_rgb}
\end{figure}
\end{document}